\documentclass[12pt, reqno]{amsart}
\usepackage{amsfonts}
\usepackage{bbm}
\usepackage{amscd,amsfonts}
\usepackage{amssymb, eucal, amsfonts, amsmath, xypic, latexsym, tikz}
\usepackage{pifont}
\usepackage{mathrsfs,color}
\usepackage{amsthm,indentfirst,bm,fancyhdr,dsfont}
\usepackage{graphicx}
\usepackage[all]{xy}
\usepackage[CJKbookmarks=true]{hyperref}

\usepackage{mathrsfs}
\usepackage{amsmath}
\usepackage{amssymb}
\usepackage{hyperref}

\newtheorem{theorem}{Theorem} [section]
\theoremstyle{definition}

\newtheorem{defn}[theorem]{Definition}%[section]
\newtheorem{example}[theorem]{Example}%[section]
\newtheorem{remark}[theorem]{Remark}%[section]
\newtheorem{prop}[theorem]{Proposition}
\newtheorem{lemma}[theorem]{Lemma}
\newtheorem{corollary}[theorem]{Corollary}
\newtheorem{conj}[theorem]{Conjecture}
\newtheorem{conven}[theorem]{Convention}

\newcommand{\Hom}{\mathrm{Hom}}

\newcommand{\al}\alpha

\newcommand{\End}{\mrm{End}}

\newcommand{\gl}{\mathfrak{gl}}

\newcommand{\om}\omega

\newcommand{\bbz}{\mathbb{Z}}
\newcommand{\bbn}{\mathbb{N}}
\newcommand{\Lie}{\mathsf{Lie}}

\def\gl{\mathfrak{gl}}

\def\cb{\mathcal{B}}
\def\gr{\text{gr}}

\def\scrp{\mathscr{P}}
\def\scrf{\mathscr{F}}

\def\jM{{{}_JM}}

\def\ggg{\mathfrak{g}}

\def\ppp{\mathfrak{p}}

\def\hhh{\mathfrak{h}}

\def\nnn{\mathfrak{n}}
\def\uuu{\mathfrak{u}}

\def\caln{\mathcal{N}}

\def\calf{\mathcal{F}}
\def\calv{\mathcal{V}}
\def\calh{\mathcal{H}}
\def\calk{\mathcal{K}}
\def\callp{{\mathcal{L}\hskip-3pt\mathscr{P}}}

\def\ucaln{\underline{\mathcal{N}}}

\def\calp{\mathscr{P}}

\def\scrf{\mathscr{F}}
\def\scraf{{}^{\textsf{a}}\hskip-3pt\mathscr{F}}
\def\scref{{}^{\textsf{e}}\hskip-3pt\mathscr{F}}
\def\pia{{{}^{\textsf{a}}\hskip-2pt\pi}_{\hskip-2pt{\lambda[q]}}}
\def\pie{{{}^{\textsf{e}}\hskip-2pt\pi}_{\hskip-2pt{\lambda[q]}}}
\def\pias{{{}^{\textsf{a}}\hskip-2pt\pi}}
\def\pies{{{}^{\textsf{e}}\hskip-2pt\pi}}
\def\scrpe{{{}^{\textsf{e}}\hskip-2pt\scrp_n}}

\def\bk{\mathbf{k}}

\def\ttF{\texttt{F}}

\def\bbf{\mathbb{F}}
\def\bbz{\mathbb{Z}}
\def\bbq{\mathbb{Q}}
\def\bbk{\mathbf{k}}
\def\bbn{\mathbb{N}}

\def\ug{\underline{g}}
\def\ugg{\underline{\mathfrak{g}}}
\def\uG{\underline{G}}

\def\uv{\underline{v}}

\def\uv{{\underline{v}}}

\def\bk{\mathbf{k}}

\def\aq{{/\hskip-4pt/}}
\def\Mm{{M^{(m)}}}
\def\Mmp{{M^{(m)}_+}}

\def\sfO{\textsf{O}}

\def\sft{\textsf{t}}

\def\Lie{\mathsf{Lie}}

\def\im{\text{im}}

\def\Ad{\mathsf{Ad}}
\def\gr{\text{gr}}
\def\GL{\text{GL}}
\def\Hom{\text{Hom}}
\def\End{\text{End}}

\def\id{\mathsf{id}}
\def\gr{\mathsf{gr}}

{\vskip-\lastskip\medskip
  \noindent
  }%
{\qed\par\medskip
  }
\newcommand{\co}{\mathcal{O}}
\def\GL{\text{\rm GL}}
\def\diag{\text{\rm Diag}}

 \usepackage{amsmath} \numberwithin{equation}{section}

\begin{document}



%\thanks{This work is partially supported by the National Natural Science Foundation %of China (12071136, 11771279, 12271345), and Shanghai Key Laboratory of PMMP (No. %13dz2260400).
%Partial results of this article were reported by the first author in a %workshop on Lie theory (Xiamen, April 2021) and  in a national conference on %Lie Theory  (Harbin, July 2021).}
%}

\title[On enhanced reductive groups (II)]{On enhanced reductive groups (II): \\nilpotent orbits under enhanced group action and their closures}
\author{Bin Shu, Yunpeng Xue and Yufeng Yao}

\begin{center}\textsc{Dedicated to the memory of Professor Guang-Yu Shen}
\end{center}

\address{School of Mathematical Sciences, Ministry of Education Key Laboratory of Mathematics and Engineering Applications \& Shanghai Key Laboratory of PMMP,  East China Normal University, No. 500 Dongchuan Rd., Shanghai 200241, China} \email{bshu@math.ecnu.edu.cn}
\address{School of Mathematical Sciences, Ministry of Education Key Laboratory of Mathematics and Engineering Applications \& Shanghai Key Laboratory of PMMP,  East China Normal University, No. 500 Dongchuan Rd., Shanghai 200241, China}
 \email{1647227538@qq.com}
\address{Department of Mathematics, Shanghai Maritime University, No.1550 Lingang Av., Shanghai, 201306, China.}\email{yfyao@shmtu.edu.cn}

\begin{abstract} This is a sequel to \cite{osy} and \cite{sxy}. Associated with $G:=\GL_n$ and its rational representation $(\rho, M)$ over an algebraically closed field $\bk$, we define an enhanced reductive algebraic group $\uG:=G\ltimes_\rho M$ which is a product variety $\GL_n\times M$, endowed with an enhanced cross product. In this paper, we first show that the nilpotent cone $\ucaln:=\caln(\ugg)$ of the enhanced Lie algebra $\ugg:=\Lie(\uG)$  has finite nilpotent orbits under adjoint $\uG$-action if and only if  up to tensors with one-dimensional modules,
  $M$ is isomorphic to one of the three kinds of modules: (i)  a one-dimensional module, (ii) the natural module $\bk^n$, (iii) the linear dual of $\bk^n$  when $n>2$; and $M$ is an irreducible module of dimension not bigger than $3$ when $n=2$. We then investigate the geometry of enhanced nilpotent orbits when the finiteness occurs.  Our focus is on the enhanced group $\uG=\GL(V)\ltimes_{\eta}V$ with the natural representation $(\eta, V)$ of $\GL(V)$, for which we give a precise classification of finite  nilpotent orbits via a finite set $\scrpe$ of  so-called enhanced partitions of $n=\dim V$, then give a precise description of  the closures of enhanced nilpotent orbits via constructing so-called enhanced flag varieties. Finally, the $\uG$-equivariant intersection cohomology  decomposition on the nilpotent cone of $\ugg$ along the closures of nilpotent orbits is established.
\end{abstract}

\maketitle

\setcounter{tocdepth}{1}\tableofcontents
\begin{center}
%\today
\end{center}
\section*{Introduction}

This note is a sequel to \cite{osy} and \cite{sxy}. The program of study of so-called enhanced reductive algebraic groups (and Lie algebras) is motivated by the study of modular representations of simple Lie algebras of non-classical type (see \cite{osy} and \cite{sxy}). We always assume that $\bk$ is an algebraically closed field.

\subsection{} Let us quickly look back on some background. In comparison with the great progress of modular representations of reductive Lie algebras in recent decades (see \cite{bm}, \cite{bmr1}, \cite{bmr2}, {\sl{etc}}), it is far away from understanding  modular representations of finite-dimensional simple Lie algebras of non-classical type  because, at least in some extent, the latter are not algebraic Lie algebras, lack of user-friendly platform. We try to find some way to cover this shortage. Let $\mathscr{L}$ be a finite-dimensional restricted  simple Lie algebras of non-classical type over an algebraically closed field of positive characteristic $p>7$, which is naturally endowed with an distinguished filtration structure $\{\mathscr{L}_i\}_{i\geq -d}$ with $d= 1,\text{ or }2$ (see \cite{ps}, or \cite{sf}).
 Then by Wilson's result \cite{wil}, $\text{Aut}(\mathscr{L})$ is a connected algebraic group  preserving both $p$-structure and the filtration structure of $\mathscr{L}$. The group $\text{Aut}(\mathscr{L})$ (denoted by $\uG$ for the time being) has a decomposition of the form $G\ltimes U$ with $G$ being a connected reductive algebraic group and $U$ the unipotent radical. The structure of $G$ and $U$ arise from the natural graded structure of $\mathscr{L}$. Roughly speaking, $\Lie(\uG)$ is actually the maximal distinguished filtered subalgebra $\mathscr{L}_0$. To a large  extent, the study of irreducible representations of $\mathscr{L}$ can be reduced to those of $\mathscr{L}_0$ (see \cite{shen},  \cite{sy}, \cite{sk}, {\sl etc}.). Such a $\uG$ of the form $G\ltimes U$ is called semi-reductive in \cite{osy}. Correspondingly, it becomes a vital topic for us to study representations of semi-reductive algebraic groups and their Lie algebras in prime characteristic.  We hope our program can shed some light on the study of modular representations of non-classical simple Lie algebras.

On the other hand, it has its own interest. Especially, when considering the enhanced product of a connected reductive algebraic group $G$ and an arbitrarily given (finite-dimensional) rational representations $(\rho, M)$, we have a natural semi-reductive algebraic group $G\ltimes_\rho M$ (see \S\ref{sec: enhanced gps}), which is called an enhanced reductive algebraic group. There are really lots of interesting new phenomena and questions related to different classical theories (see \cite{osy}, \cite{sxy} and \cite{Xue}).

\subsection{} In this paper, we focus on the enhanced reductive algebraic group $\uG:=G\ltimes_\rho M$ and $\ugg:=\Lie(\uG)$ for  $G=\GL_n$ and finite-dimensional rational representations $(\rho,M)$ of $G$ over $\bk$. Especially, we investigate the nilpotent cone $\ucaln:=\caln(\ugg)$, and the $\uG$-conjugacy classes in $\ucaln$. In view that the question can be reduced to the case when $(\rho,M)$ is a polynomial representation,  so we will focus our interest in this case where  irreducible modules of $\GL_n$ can be parameterized by the set  $P^+$ of partitions of positive integers, up to isomorphism (see \cite{Gr}). So we can write them as $\{L(\lambda)\mid \lambda \in P^+\}$.

Recall that in the classical theory on representations of reductive algebraic groups and their Lie algebras, the geometry concerning the nilpotent cones and their nilpotent orbits is very important (see \cite{jan2}). Their  geometric features  reflect rich representation information along with combinatorics properties. Naturally, it can be expected that for $\uG$, there are rich geometric properties with deep connection to representations of $\Lie(\uG)$ in prime characteristic. As further perspectives, one can expect to explore the approach in modular representations of reductive Lie algebras in the study of representations of $\mathscr{L}_0$, thereby of non-classical simple Lie algebra $\mathscr{L}$.

\subsection{Main results} The first purpose of this paper is to determine when $\ucaln$ mentioned above has finite orbits.  As is well-known, the nilpotent cones for reductive Lie algebras over $\bk$ have finite orbits (see for example \cite[Theorem 2.8.1]{jan2}).  This primary property  is a basis of geometric theory of reductive Lie algebras and their representations. It is worthwhile to solve the finiteness problem in the enhanced Lie algebra case. In this part, our main result is {as below.}

\begin{theorem}\label{thm: zerosec thm} (see Theorems \ref{thm: bigg n orbit} and \ref{thm: gl2 enh class}) Let $G=\GL_n$ and $(\rho,M)$ be a finite-dimensional rational representation over $\bk$ of any characteristic.
%Suppose $\text{ch}(\bk)=0$, or loosely $\text{ch}(\bk)$ is any while $M$ has complete reducibility.
Keep the notations as above. In particular,  $\uG=\GL_n\ltimes_\rho M$, and $\ugg:=\Lie(\uG)$ which is isomorphic to $\gl_n\times M$ as a vector space.
\begin{itemize}
\item[(1)] If $\ucaln$ has finite orbits, then $M$ is an irreducible module.
\item[(2)] When $n>2$, the nilpotent cone $\ucaln$ of the enhanced Lie algebra $\ugg$  has finite nilpotent orbits under adjoint $\uG$-action if and only if this irreducible module $M$, up to tensor with a one-dimensional module, is isomorphic to one of the following three modules: (i) one-dimensional module, (ii) the natural module $\bk^n$;  (iii) the linear dual of natural module.
        \item[(3)] When $n=2$ and $\text{ch}(\bk)=0$, the nilpotent cone $\ucaln$ has finite nilpotent orbits if and only if the irreducible module $M$ is of dimension not bigger than $3$.
\end{itemize}
\end{theorem}
The basic observation here is a fact that for any $(X,w)\in\ucaln$, $(X,w+\im X)\in \textsf{Ad}(G)(X,w)$ where $\im X$ stands for the image of $\textsf{d}(\rho)X$ in $M$ (see Lemma \ref{lem A}), consequently the analysis of the $G_X$-action on $M$ (modulo $\im X$) becomes an important ingredient in the proof of Theorem \ref{thm: zerosec thm} where $G_X$ is the centralizer  of $X$ in $G$.
%Another point is that %the analysis on the centralizer of $X$ in $G$ will be important in our %arguments.
The technique in most parts of the proof of the above theorem is elementary except the $\GL_2$ case  where we need to verify the existence of nonzero dimension of an algebraic quotient associated with any given polynomial irreducible representation of dimension bigger than $3$, in order to show the infiniteness of orbits. To achieve this, we take use of some classical geometric arguments (see Lemma \ref{lem: gl2 homo dim}).

\subsubsection{Exceptional case} Then we turn to the study of $\ucaln$ when it has finite orbits.
In Theorem \ref{thm: zerosec thm}(3), $\ucaln$ has finite nilpotent orbits when it is associated with $\GL_2$ and a $3$-dimensional module $M$. In the second section, we give a precise description of all nilpotent orbits and closures (see Proposition \ref{prop: 3d prop}).

\subsubsection{Enhanced nilpotent orbits associated with natural modules}
%The basic observation via Fourier transforms enables us to
Some basic computations enables us to
establish some close relation between the nilpotent orbit theories associated with linear dual modules presented in (iii) of  Theorem \ref{thm: zerosec thm}(2) and associated with natural modules presented there (ii)  (see \S\ref{sec: Fourier trans}). In the present paper, for further investigation of nilpotent orbits we are only concerned with the case when the enhanced structures are associated with natural modules.  As to the case of dual modules,  one can expect to some interesting results in parallel, which is worthwhile to be investigated somewhere else.

The exceptional case occurs only when $G=\GL_2$, and $M$ is a $3$-dimensional irreducible module. We will specially deal with it in \S\ref{sec: exceptional case}.
 The remaining main body of this paper %except the last section
 is devoted  to the investigation of enhanced nilpotent orbits and their closures for $\uG=\GL(V)\ltimes_\eta V$ where $(\eta,V)$ is the natural representation.

 The significant extra information of the $\uG$-conjugacy classes in $\ucaln$,  apart from the control of $\GL(V)$-conjugation, comes from the mixed information from the representation on $V$, especially from $\im\textsf{d}\eta(X')\subset V$ for $X'\in \textsf{Ad}(G)X$ with $G=\GL(V)$. Precisely speaking, the $\uG$-orbits in $\ucaln$ can be parameterized by a set $\scrpe$  of all so-called enhanced partitions of $n$, which can be described as  pairs $(\lambda,[q])$ with $\lambda$ running through the set $\scrp_n$ of all partitions of $n$ ($n=\dim V$), being attached with its so-called $q$th lowerings (see Theorem \ref{prop: altn nilp cls}, and \S\ref{sec: lowerings} for the notion ``lowerings"). The pairs $(\lambda,[q])$ are simply written as $\lambda[q]$ in the whole article, the latter of which is  nominated  an enhanced partition of $n$.
 As can be expected,  the closures of these $\uG$-orbits $\co_{\lambda[q]}$ are controlled by the original $\GL_n$-orbits along with the ``lowerings" of partitions (see Corollary \ref{cor: closure M}),  with topological relation determined by some natural order of enhanced partitions (see Proposition \ref{prop: closure rel}).

Along this line, the classical theory  of nilpotent orbits and Springer theory can be extended in the frame of the enhanced reductive groups. Especially, there are ``enhanced" partial flag varieties and ``enhanced" Springer resolutions so that ``enhanced Springer fibers" can be expected to play key roles in the study of  modular representations of enhance reductive Lie algebras, and further go into that of $\mathscr{L}$ although the enhanced Springer fibers turn out to coincide with the corresponding classical Springer fibers in the case $\uG=\GL(V)\ltimes_{\eta}V$. In this note, we take a look into the enhanced Springer fibers and their cohomology, finally demonstrating  a $\uG$-equivariant intersection cohomology decomposition on $\ucaln$ along the closures of their $\uG$-orbit decomposition (see Theorem \ref{thm: 4.1}).

\subsection{Relation to the work on $\GL(V)$-orbits in $\caln\times V$} This work is much related to the earlier works of Achar-Henderson and of Travkin ({\sl etc.}) when $M$ is the natural module (see \cite{AH}, \cite{FGT}, \cite{Tr}). In \cite{AH} and \cite{Tr}, Achar-Henderson and Travkin independently investigated the $\GL(V)$-orbits on $\caln\times V$, with motivations coming from the consideration in type $A$ of Syu Kato's work on  ``exotic Springer correspondence" in type $C$, and the study involving  some mixed flag varieties, respectively. Their parameterizations of $\GL(V)$-orbits are the set of bi-partitions of $n$. Our parameterizations of $\uG$-orbits are closely related to and compatible with theirs (see \S\ref{sec: comparison}).

%\subsection{} 		
	\subsection*{\textsc{Acknowledgement}}	This work is partially supported by the National Natural Science Foundation of China (Grant No. 12071136, and 12271345), and  by Science and Technology Commission of Shanghai Municipality (No. 22DZ2229014).  The authors are thankful to the anonymous referee for helpful comments.
%BS expresses his deep thanks to Shun-Jen Cheng for his helpful discussions and his introduction of the references \cite{ChKac98, ChKac99, Kac98}.

\section{Finiteness of  nilpotent orbits under the enhanced group action}
Throughout the paper,  all algebraic groups and  vector spaces are assumed to be over an algebraically closed {field $\bk$ unless stated otherwise.}

\subsection{Enhanced algebraic groups and enhanced Lie algebras}\label{sec: enhanced gps}
Let $G$ be a connected algebraic group over $\bk$, and $(M,\rho)$ be a finite-dimensional representation of $G$ on the vector space  $M$.  Consider the product variety $G\times M$. Regard $M$ as an additive algebraic group. The irreducible variety $G\times M$ is endowed with a cross product structure denoted by $G\times_\rho M$,  by defining for any $(g_1,v_1), (g_2,v_2)\in  G\times M$
\begin{equation}\label{product}
(g_1,v_1)\cdot (g_2,v_2):=(g_1g_2, \rho(g_1)v_2+v_1).
\end{equation}
Then a straightforward computation yields that  $\underline{G}:=G\times_\rho M$ becomes an algebraic group with unity $(e,0)$ for the unity $e\in G$, and $(g,v)^{-1}=(g^{-1}, -\rho(g)^{-1}v)$.  Both $G$ and $M$ naturally become closed subgroups of $\underline{G}$, and $M$ is normal. { When $G$ is reductive,} the unipotent radical of $\uG$ is just the closed subgroup $e^M:=\{e\}\times M$, which is identified with $M$ (an additive algebraic group).  We call $\underline{G}$ \textsl{ an enhanced algebraic group associated with the representation space $M$}.
Set $\ggg=\Lie(G)$. Then  $(M, \mathsf{d}(\rho))$ becomes a representation of $\ggg$.  And $\underline{\ggg}:=\Lie(\underline{G})=\ggg\times_{\textsf{d}\rho}M\cong \ggg\oplus M$ with the following Lie bracket
$$[(X_1,v_1),(X_2,v_2)]:=([X_1,X_2], \mathsf{d}(\rho)(X_1)v_2-\mathsf{d}(\rho)(X_2)v_1),$$
which is called an enhanced Lie algebra. Naturally, $\underline{\ggg}$ is an adjoint $\underline{G}$-module with
\begin{equation}\label{Ad structure}
\textsf{Ad}(g, v)(X, w)=(\textsf{Ad}(g)X, {-\textsf{d}\rho}(\textsf{Ad}(g)X)v+\rho(g)w),\,\,\forall\,g\in G, X\in\ggg, v, w\in M.
\end{equation}
In the following, we always simply write $X. v$ or $Xv$ for $\textsf{d}\rho(X)v$, and $g. v$ or $gv$ for $\rho(g)v$ for $X\in \ggg, g\in G$ and $v\in M$ if the context is clear.

From now on, we always assume that $G=\GL_n$ is the general linear algebraic group of rank $n$, and $M$ is a finite-dimensional rational $G$-module. Let
$\caln$ and $\ucaln$ be the nilpotent cones of $\ggg=\gl_n$ and the nilpotent cone of $\underline{\ggg}=\gl_n\times M$ { (the latter of which is  by definition the set of   nilpotent elements in $\underline{\ggg}$ too)}, respectively. Then
\begin{align}\label{eq: gen enh nil cone}
\ucaln=\caln\times M.
\end{align}
The following fact is fundamental.

\begin{lemma}\label{lem: basic lem} Let $G=\GL_n$ and  $M\cong\bk^n$ be the natural module of $G$, and $\uG$ the associated enhanced reductive group.  For any $(X,w)\in \ucaln$, denote $\ugg_{(X,w)}=\{(Y,u)\in \gl_n\times M\mid \textsf{ad}(X,w)(Y,u)=0\}$ and $\uG_{(X,w)}=\{(g,v)\in \uG\mid \Ad(g,v)(X,w)=(X,w)\}$. Then the following statements hold.
\begin{itemize}
\item[(1)] $\ugg_{(X,w)}=\Lie (G_{(X,w)})$
\item[(2)] The centralizer $\uG_{(X,w)}$ is connected.
\end{itemize}
\end{lemma}
\begin{proof}
Set $\ggg=\gl_n$, and denote by $\ggg_X$ and $G_X$ the centralizers of $X$ in $\ggg$ and in $G$ respectively. By definition, we can regard $\uG$ as a connected closed subgroup of $\GL_{n+1}$, and naturally $\uG_{(X,w)}=\uG\cap {\GL_{n+1}}_{(X,w)}$. Recall a basic fact for general linear groups (see, for example \cite[\S2.3]{jan2}): $\Lie({\GL_{n+1}}_{(X,w)})={\gl_{n+1}}_{(X,w)}$ for any $(X,w)\in\ugg$. So we have $\Lie(\uG_{(X,w)})\subset \ugg_{(X,w)}$. Note that $\uG_{(X,w)}$ is actually the principal open subvariety of $\ugg_{(X,w)}$  by $\det$. Hence  the centralizer $\uG_{(X,w)}$ is irreducible, yielding  the statement (2). What is more, $\dim\uG_{(X,w)}=\dim\ugg_{(X,w)}$. So $\dim\Lie(\uG_{(X,w)})=\dim\ugg_{(X,w)}$, correspondingly $\Lie(\uG_{(X,w)})=\ugg_{(X,w)}$.  Hence (1) is proved.
\end{proof}

\subsection{A {criteria Lemma on the finiteness of nilpotent orbits in} enhanced Lie algebras}
We first look at what the nilpotent orbits look like in the enhanced Lie algebra $\underline{\ggg}$. For any given nilpotent element $(X,w)\in \ucaln$ we suppose the Jordan standard matrix of $X$ under $\GL_n$-conjugation is $J=J_\lambda$ for $\lambda=(\lambda_1\geq \lambda_2\geq\ldots \geq\lambda_r>0)$ with $\lambda\in \mathscr{P}_n$ for some $r\in \bbn$, where $\scrp_n$ denotes the set of all partitions of $n$, i.e.
$$\mathscr{P}_n:=\Big\{\mu=(\mu_1\geq \mu_2\geq\ldots \geq\mu_m>0)\mid m\in\mathbb{N}, \sum_{i=1}^m\mu_i=n\Big\}$$
with $\aleph_n:=|\mathscr{P}_n|$.
This means that
\begin{equation*}
J=\left( \begin{array}{cccc}
 J_{\lambda_1} & \empty &\empty &\empty \cr
\empty &J_{\lambda_2}   &\empty &\empty\cr
\empty &\empty &\cdots &\empty \cr
\empty&\empty&\empty& J_{\lambda_r}
\end{array}\right),
\end{equation*}
where for $i=1,\ldots, r$
\begin{equation*}
J_{\lambda_i}=\left( \begin{array}{ccccc}
 0 &1 & 0 &\empty &\empty   \cr
\empty &0 &1 &\empty &\empty\cr
\empty&\empty&\cdots& \empty&\empty \cr
\empty&\empty&\empty& 0&1\cr
\empty&\empty&\empty&\empty&0
\end{array}\right)_{\lambda_i\times\lambda_i}.
\end{equation*}

Set $G_J:=\{g\in G\mid \textsf{Ad}(g)J=J\}$, the centralizer of $J$ in $G$. Then $G_J$ is a closed subgroup of $G$. Parallel to $G_J$, we have $\ggg_J:=\{Y\in \ggg\mid [Y,J]=0\}$.
Set $\jM=\textsf{d}\rho(J)(M)$. Then $\jM$ is a $G_J$-module. (In the same sense, we have the notations $G_X$ and $\ggg_X$ for any $X\in \ggg$ with the same meaning as above.)

We first have the following {criteria on the finiteness of enhanced nilpotent orbits, which is critical important in the subsequent arguments}.

\begin{lemma}\label{lem A} {(Criteria Lemma)}   Keep the notations as above. The following statements hold.
\begin{itemize}
\item[(1)] For any $(J,w)\in \ucaln$, $(J, w+\jM){\subset} \textsf{Ad}(\uG)(J,w)$.
\item[(2)]Denote $\overline \jM=M\slash \jM$. Then $\uG$-orbits in $\ucaln$ are finite if and only if $G_J$-orbits in $\overline\jM$ are finite for {all} $J=J_\lambda$ with  $\lambda\in \mathscr{P}_n$.
\end{itemize}
\end{lemma}

\begin{proof} (1) It is directly deduced from the definition.

(2) Recall that the set of $G$-orbits in $\caln$ is finite, and it is one to one correspondence to $\mathscr{P}_n$. For  $J=J_\lambda$ with  $\lambda\in \mathscr{P}_n$, let $\Upsilon_{\lambda}$ be the set of representatives in $M$ of the $G_J$-orbit in $\overline\jM$. Then it follows from  (\ref{Ad structure}) that we have the following bijection:
$$
\aligned
\Phi: \,\, \{(\lambda, w_{\lambda})\mid \lambda\in\mathscr{P}_n, w_{\lambda}\in \Upsilon_{\lambda}\}
&\longrightarrow \ucaln/  \uG \cr (\lambda, w_{\lambda}) &\mapsto
 \textsf{Ad}(\uG) (J_{\lambda}, w_{\lambda}).
\endaligned
$$
This implies that $\uG$-orbits in $\ucaln$ are finite if and only if $\Upsilon_{\lambda}$ is finite for any $\lambda\in \mathscr{P}_n$, as desired.
\end{proof}

\subsection{Natural and dual natural representations}\label{sec: Fourier trans} For $V=\bk^n$, there is a standard basis {$\mathfrak{E}:=\{e_i\}_{i=1,\ldots,n}$}. Naturally, the linear dual space $V^*$ admits the dual basis $\{e_i^*\}_{i=1,\ldots,n}$ defined via $e_i^*(e_j)=\delta_{ij}$, $i,j=1,\ldots,n$. {Here and further $\delta_{ij}$ denotes the Kronecker delta function, which means it takes value $0$ if $i\ne j$, and $1$ if $i=j$.}  For any $w\in V$, it can be expressed $w=\sum_{i=1}^n a_ie_i$ or $w=(e_1,\ldots,e_n)\alpha$ for $\alpha=(a_1,\ldots,a_n)^{\texttt{t}}$ where $\texttt{t}$ stands for the transpose of matrices.

{
Then we  introduce a Fourier transform $\scrf$ from $V$ to $V^*$ which is defined via sending $w=\sum_{i=1}^n a_ie_i$ to $\scrf(w):=\sum_{i=1} a_ie_i^*$.  By definition, $\scrf$ is a linear isomorphism.  }

%More generally, for any basis $\uv:=\{v_i\}_{i=1,\ldots,n}$ of $V$, one can introduce a Fourier transform $\scrf_\uv: V\rightarrow V^*$ with respect to $\uv$. So $\scrf$ %is just identical to $\scrf_{\mathfrak{E}}$.

\iffalse
Denote by $E$ the transition matrix from {$\mathfrak{E}$} to $\uv$, i.e.
$$(v_1,\ldots,v_n)=(e_1,\ldots, e_n)E.$$
Then $E$ is an invertible matrix of size $n$.
For a square matrix $C$ of size $n$, by abuse of notation we still denote by $C$ the linear transform of $V$ sending $w=(e_1,\ldots,e_n)\alpha^{\texttt{t}}$ to $(e_1,\ldots, e_n)C\alpha^{\texttt{t}}$.
Then it is easily verified that

\begin{equation}\label{eq: indep of Foure}
\scrf_\uv=\scrf\circ (EE^{\texttt{t}})^{-1}.
\end{equation}
\fi

For $G=\GL_n$, we have the natural representation $\eta$ on $V$ and its dual representation $\eta^\sharp$ on $V^*$. Consider the enhanced groups $\uG=G\ltimes V$ and $\uG^\sharp:=G\ltimes V^*$. The corresponding Lie algebras are denoted by $\ugg$ and $\ugg^\sharp$ respectively. {For convenience of arguments, we suppose $V=\bbk^n$ with a standard basis $(e_1,\ldots,e_n)$. With this basis,   $\text{End}_\bbk(V)$ is in a one-to-one correspondence with the matrix ring $\text{Mat}_{n\times n}(\bbk)$ of size $n\times n$ over $\bbk$. This means, for any $\sigma\in \text{End}(V)$, we can write $\sigma(e_1,\ldots,e_n)=(e_1,\ldots,e_n)A_\sigma$ for $A_\sigma\in \text{Mat}_{n\times n}$. So we can talk $\sigma^{\texttt{t}}$ which means the transformation corresponding to $A_\sigma^{\texttt{t}}$.
In this setup, the following observation is clear.}

\begin{lemma}\label{lem: Fourier transform} Consider the enhanced groups   $\uG$ and $\uG^\sharp$. The following statements hold.
\begin{itemize}

\item[(1)] For any $g\in G$, {$\eta^\sharp(g)\circ \scrf=\scrf\circ \eta(g^{-\texttt{t}})$ in $\Hom_{\bk}(V,V^*)$. For any $X\in \gl_n$,
    $\textsf{d}(\eta^\sharp)(X)\circ\scrf=-\scrf\circ \textsf{d}(\eta)(X^\texttt{t})$ in  $\Hom_{\bk}(V,V^*)$.  More precisely,}
for $X\in\gl_n$ and $g\in \GL_n$,  the following equations hold:
    \begin{align}\label{eq: dual formula 1}
    \eta^\sharp(g)\scrf(v)=\scrf(\eta(g^{-\texttt{t}})v)
    \end{align}
    and
    \begin{align}\label{eq: dual formula 2}
    \textsf{d}(\eta^\sharp)(X)\scrf(v)=\scrf(\textsf{d}(\eta)(-X^{\texttt{t}}v)) \text{ with } v\in V. \end{align}

\item[(2)] For any two elements $(X_i,w_i)\in \ugg$ ($i=1,2$), $(X_1,w_1)$ lies in the $\uG$-orbit of $(X_2,w_2)$ if and only if  {$(-X^{\texttt{t}}_1,-\scrf(w_1))$ lie in the $\uG^\sharp$-orbit of $(-X^{\texttt{t}}_2,-\scrf(w_2))$.}
\iffalse
\item[(3)]  The above property (2) is independent of the choice of $\scrf_\uv$, up to translation by $\mathcal{E}:=EE^{\texttt{t}}$ for a transmit matrix $E$. This is to say, for any Fourier transform $\scrf_\uv$,
      $(X_1,w_1)$ lies in an $\uG$-orbit of $(X_2,w_2)$ if and only if $(X_1, \scrf_\uv( \mathcal{E}w_1))$ lie in a $\uG^\sharp$-orbit of $(X_2, \scrf_\uv( \mathcal{E}g^2w_2))$ for  a suitable $g\in G$.
      \fi
\end{itemize}
\end{lemma}

\begin{proof}
(1) { By the definition of  linear dual modules over algebraic groups, the first part is directly deduced. We need to verify the second part. We still assume $V=\bbk^n$ with a canonical basis $(e_1,\ldots,e_n)$. For any given $v\in V$ , we can write
\begin{equation*}
v=(e_1,\ldots,e_n)\left( \begin{array}{c}a_1\cr a_2\cr \vdots\cr a_n
\end{array}\right).
\end{equation*}
Furthermore,
%For $g\in \GL_n$ and $X\in \gl_n$,
we can write $g(e_1,\ldots,e_n)=(e_1,\ldots,e_n)A_g$, and $X(e_1,\ldots,e_n)=(e_1,\ldots,e_n)A_X$. Note that $\scrf(e_i)(e_j)=\delta_{ij}$, and $\scrf$ is a linear isomorphism.
{For any $w\in V$} with $w=(v_1,\ldots,v_n)(b_1,\ldots,b_n)^{\texttt{t}}$, we then have
{
\begin{align*}
&\eta^\sharp(g)\scrf(v)(w)\cr
=&\scrf(v) (\eta(g^{-1})w)\cr
=&(\scrf(e_1,\ldots,e_n)(a_1,\ldots,a_n)^{\texttt{t}})
((e_1,\ldots,e_n)A_g^{-1}(b_1,\ldots,b_n)^{\texttt{t}})\cr
=&(a_1,\ldots,a_n) (A_g^{-1}(b_1,\ldots,b_n)^{\texttt{t}})\cr
=&(\scrf(e_1,\ldots,e_n)A_g^{-\texttt{t}}(a_1,\ldots,a_n)^{\texttt{t}})((e_1,\ldots,e_n)
(b_1,\ldots,b_n)^{\texttt{t}})\cr
=&\scrf(\eta(g^{-\texttt{t}})v)w
\end{align*}
}
from which we deduce  (\ref{eq: dual formula 1}). Similarly, one can prove (\ref{eq: dual formula 2}).
}

(2) Suppose there exists $(g,v)\in \uG$ such that $(X_1,w_1)=\texttt{Ad}(g,v)(X_2,w_2)$.  With aid of the statement (1), by computation we have {$(-X^{\texttt{t}}_1,-\scrf(w_1))=\texttt{Ad}
(g^{-\texttt{t}},-\scrf(v))(-X^{\texttt{t}}_2,-\scrf(w_2))$.}  Note that $\scrf$ is a linear isomorphism. Hence the reverse is also true.
\iffalse
(3) By (\ref{eq: indep of Foure}) and the above second statement, we have that $(X_1,w_1)=\texttt{Ad}(g,v)(X_2,w_2)$ for $(g,v)\in \uG$ if and only if $(X_1, \scrf_\uv(\mathcal{E}w_1))=
\texttt{Ad}(g,-\scrf_\uv(\mathcal{E}v))(X_2,\scrf_\uv(\mathcal{E}g^2w_2))$.
\fi
The proof is completed.
\end{proof}

\begin{remark}\label{rem: identical orbti thy}
(1) Note that $\caln(\ugg^\sharp)=\caln(\ggg)\times V^*$ and $\caln(\ugg)=\caln(\ggg)\times V$ (see (\ref{eq: gen enh nil cone})). {
 %By  Lemma \ref{lem: Fourier transform}(2),
  %  $$\varphi: (X,w) \mapsto (X^{\texttt{t}}, \scrf(w)).$$
 With aid of the above lemma,  it's easily known by a straightforward computation that there is an algebraic group isomorphism $\maltese$ from $\uG$ onto $\uG^\sharp$ defined via
 $$\maltese: (g,v)\mapsto (g^{-\texttt{t}}, -\scrf(v)).$$
From $\maltese$, there is an induced isomorphism $\mho$ of Lie algebras from $\ugg$ to $\ugg^\sharp$ defined via
$$\mho: (X,v)\mapsto (-X^{\texttt{t}}, -\scrf(v)).$$
This $\mho$ carries  $\caln(\ugg)$ onto $\caln(\ugg^\sharp)$.

(2)  Consequently, we can say such $\mho|_{\caln(\ugg)}$ (denoted by $\varphi$) is a $\uG$-equivariant isomorphism in the sense that the isomorphism $\varphi$ of varieties  satisfies
$$\varphi( (g,u)(X,v))=\maltese((g,u))\varphi((X,v))$$
for $(g,u)\in \uG$.
 % $\close connection between the orbit theory of $\caln(\ugg^\sharp)$ under %$\uG^\sharp$-action and the one of  $\caln(\ugg)$ under $\uG$-action. %{The first observation is the following:
  %\begin{observation}
 Hence, there is a one-to-one correspondence  between the set of $\uG$-orbits in $\caln(\ugg)$ and the set of $\uG^\sharp$-orbits in $\caln(\ugg^\sharp)$.}
% \end{observation}

 \iffalse
  Actually, by Lemma \ref{lem: Fourier transform}(2) and its proof, if two elements $(X_1,w_1)$ and $(X_2,w_2)$ in $\caln(\ugg)$ are $\uG$-conjugate, say,
 $(X_1,w_1)=\texttt{Ad}(g,v)(X_2,w_2)$ for $(g,v)\in \uG$, then $$(X_1,\scrf(w_1))=\texttt{Ad}(g,-\scrf(v))(X_2,\scrf(g^{2}w_2)).$$
 Consequently, $(X_1,\scrf(w_1))$ and $(X_2,\scrf(g^{2}w_2))$ are $\uG^\sharp$-conjugate.  Conversely, if two elements $(X_1,\scrf(w_1)$ and $(X_2,\scrf(w_2))$ in $\caln(\ugg^\sharp)$ are $\uG^\sharp$-conjugate,
  say,
 $$(X_1,\scrf(w_1))=\texttt{Ad}(g,\scrf(v))(X_2,\scrf(w_2))$$ for $(g,\scrf(v))\in \uG^\sharp$, then $(X_1,w_1)=\texttt{Ad}(g,-v)(X_2,g^{-2}w_2))$.
 Hence $(X_1,w_1)$ and $(X_2,g^{-2}w_2)$ in $\caln(\ugg)$ are $\uG$-conjugate.
\fi

% \end{proof}

\end{remark}

\subsection{Finiteness of  nilpotent orbits under the enhanced group action} \label{sec: general fininteness}
In the next subsections, we investigate the finiteness problem for the enhanced nilpotent cones associated with $\GL_n$ and a rational module $\GL_n$-module $M$.

\subsubsection{}
Note that $M$ has a finite length of composition series. Denote this length by $l:=l(M)$. Then we can express the corresponding element in the Grothendieck ring of the category of finite-dimensional rational $G$-modules as: $[M]=\sum_{i=1}^l [L(\mu^{(i)})]$ with $\mu^{(i)}$ being a highest weight in the set of dominant integral weights such that $L(\mu^{(1)})$ is an irreducible  submodule of $M$ and $L(\mu^{(l)})$ is an irreducible quotient of $M$. For example, when $\text{ch}(\bk)=0$, $M$ satisfies the complete reducibility, for which the choice of $L(\mu^{(1)})$  and $L(\mu^{(l)})$ is very flexible.

In the following arguments, denote $J=J_n$.
Set $E=\diag(1,\ldots,1)$, the identity matrix.
Then $G_J=\textbf{G}_m\ltimes \textbf{R}$ where $\textbf{G}_m$ is generated by $g_c:=cE$ with $c(\ne 0)\in \bk$, and $\textbf{R}$ is a unipotent closed subgroup of $\GL_n$  generated by
$g^+(k,c):=E+ c\sum_{i=1}^{n-k}E_{i,k+i}$ with $k$ running over $\{1,\ldots,n-1\}$, and $c (\ne 0)\in\bk$.

 \begin{lemma}\label{lem B} %Suppose $M$ has the complete reducibility.
 If $\ucaln$ has finite orbits, then $M$ must be an irreducible module of $\GL_n$.
  \end{lemma}

  \begin{proof} {We prove the lemma by contradiction.} Suppose $M$ is not irreducible. Then the length $l>1$. As above, we suppose that $M$ has composition factors  $\{L(\mu^{(i)})\}_{i=1,\ldots,l}$.  Recall that $\ucaln=\caln\times M$. And $L(\mu^{(i)})$ has a minimal weight vector $v_0^{(i)}$, i.e, $L(\mu^{(i)})=U(\nnn^+)v_0^{(i)}$. {Here and further, $\nnn^+$ denotes the Lie subalgebra of $\gl_n$ consisting all strict upper triangular matrices}. By the choice,  $M$ has an irreducible  submodule $L(\mu^{(1)})$,  and an irreducible quotient $L(\mu^{(l)})$. So there is a natural homomorphism from $M$ onto $L(\mu^{(l)})$. Take a weight vector $u_0^{(l)}\in M$ such that  $u_0^{(l)}$ is a preimage of $v_0^{(l)}$ in $M$. For unification of notations, denote $u_0^{(1)}=v_0^{(1)}$ which belongs to $M$.

  Consider $\lambda=(n)\in\mathscr{P}_n$, and $J=J_\lambda$, which is equal to $J_n$, the Jordan standard matrix of regular $G$-nilpotent elements for $\ggg$.  %Now our arguments are divided into different cases.
  Then there exists a nonzero $\chi\in\bbk$ such that $\chi^{m}$ are different when $m$ varies in $\bbn$ if $\bbk$ is either of characteristic $0$, or of characteristic $p>0$ but not identical to $\overline{\bbf}_p$. And for any $m\in \bbn$ we can take a primitive element $\chi^{(m)}$ of $\bbf_{p^m}\subset\bbk$ if $\bbk=\overline{\bbf}_p$. Obviously, those $\chi^{(m)}$ are different when $m$ varies in $\bbn$.

By abuse of notations, for the above fixed  $\chi$ or $\chi^{(m)}$ we make the following appointment
  \begin{align}\label{eq: frobeinius notations}
  \texttt{F}_m(\chi)=\begin{cases}
  \chi^{m} &\text{ if } \bk\ne \overline{\bbf}_p;\cr
  \chi^{(m)} &\text{ if } \bk=\overline{\bbf}_p.
  \end{cases}
  \end{align}

Consider $\textsf{u}_m:=u_0^{(1)}+ \texttt{F}_m(\chi)u_0^{(l)}$ for $m\in \bbn$. It follows from (\ref{Ad structure}) that all $(J, \textsf{u}_m$), $m\in\bbn$, lie in different $\uG$-orbits by Lemma \ref{lem A}. In fact, if $\textsf{u}_{m_1}$ and $\textsf{u}_{m_2}$ lie in the same $\uG$-orbit, then it follows from (\ref{Ad structure}) that there exists $g\in G_{J}$ such that
  \begin{align}\label{eq: pre claim m1m2}
  g.(u_0^{(1)}+ \texttt{F}_{m_1}(\chi)u_0^{(l)})\equiv u_0^{(1)}+ \texttt{F}_{m_2}(\chi)u_0^{(l)}\,\,\text{mod}\,(\jM).
  \end{align}
  {With aid of this formula, we claim
 % yielding that
  \begin{align}\label{eq: m1 eqaul m2}
m_1=m_2.
  \end{align}
  Now we prove it by comparing the lowest weight vectors of summands in both sides of (\ref{eq: pre claim m1m2}). % which are linearly independent.
  Note that $G_J=\{\sum_{i=0}^{n-1}a_0J^i\mid a_0\in\bbk^\times, a_i\in \bbk\}$. Here and further $\bk^\times:=\bk\backslash\{0\}$.  We can write $g$ appearing in  (\ref{eq: pre claim m1m2}) as $g=a_0E+\sum_{i\geq 1}a_iJ^i$ for $a_0\in \bbk^\times$.  Thus, from (\ref{eq: pre claim m1m2}) with  such $g$ we deduce the following equation
 $$(a_0-1)u_0^{(1)}+(a_0F_{m_1}(\chi)-F_{m_2}(\chi))u_0^{(l)}=0$$
 in the module $\overline{\jM}$ over $G_J$ (see Lemma \ref{lem A} for the notation). Note that $u_0^{(1)}$ and $u_0^{(l)}$ are still linearly independent because both $u_0^{(1)}$ and $u_0^{(l)}$ are preimages of $v_0^{(1)}$ and $v_0^{(l)}$, respectively, latter of which are lowest weight vectors. So we have $a_0=1$ and $F_{m_1}(\chi)=F_{m_2}(\chi)$. This leads to the equality $m_1=m_2$. The claim (\ref{eq: m1 eqaul m2}) is proved. Consequently, we can conclude that all $(J,u_m)$ with $m\in \bbn$,  lie in mutually different orbits. It contradicts the preassumption of the lemma.
}

  The proof is completed.
  \end{proof}

\subsubsection{} By the above lemma, we focus on irreducible modules $M$ of $\GL_n$.  Up to tensors with one-dimensional modules, rational irreducible modules of $\GL_n$ can be presented as  polynomial irreducible ones which are parameterized by the set $P^+$ of partitions of positive integers, denoted by $L(\mu)$ with $\mu\in P^+$.
 % i.e. $M\cong L(\mu)$  with $\lambda=\sum_{i=1}^ra_i\epsilon_i$ for %$a_i\in\bbn$ and $a_i\geq a_{i+1}$.
 Note that for any one-dimensional module $\bk_\gamma$ with $\gamma\in \bk^\times$, the $\GL_n$-conjugacy classes in $M$ and in $M\otimes_\bk \bk_\gamma$  are in a one-to-one correspondence.
Consequently, for simplicity of arguments we might as well suppose $M=L(\mu)$ is a polynomial irreducible representation without loss of generality. Those $\mu$ are  partitions of some positive integers, i.e. $\mu=(a_1\geq a_2\geq \cdots)$ with  $a_i\in \bbn$, $i=1,2,\ldots$.
  What is more, $\mu$ can be written as $\sum_{i=1}^na_i\epsilon_i$ where $\epsilon_i$ stands for an element of the character group $X(\textsf{T})$ of the canonical maximal torus $\textsf{T}$ of $\GL_n$ consisting of all invertible diagonal matrices, by defining
  via $\epsilon_i(\diag(t_1,\cdots,t_n))=t_i$. Denote $\alpha_i=\epsilon_i-\epsilon_{i+1}$ for $i=1,\ldots,n-1$. Then $\Delta:=\{\alpha_i\mid i=1,\ldots,n-1\}$ becomes the system of standard simple roots for $\GL_n$, and $\Phi^+:=\{\epsilon_i-\epsilon_j\mid 1\leq i<j\leq n\}$  the system of standard  positive roots. {There is a $\mathbb{R}$-bilinear form $(\cdot,\cdot)$ on $X(T)_\mathbb{R}=X(T)\otimes_\bbz\mathbb{R}$
  defined via $(\epsilon_i,\epsilon_j)=\delta_{ij}$.}

    Furthermore, $\textsf{d}(\epsilon_i)$ becomes a linear function on the space $\hhh:=\Lie(T)$ (consisting of all diagonal matrices of $\gl_n$), taking the value $c_i$ at $\diag(c_1,\cdots,c_n)\in\hhh$. All these $\textsf{d}(\epsilon_i)$ are really linear functions on $\hhh$. Sometimes, we do not distinguish  $\epsilon_i$ and $\textsf{d}(\epsilon_i)$ if the context is clear. Regard $\{\alpha_i\mid i=1,\ldots, n-1\}$ as the simple root system of $\ggg$. In the following, we will take $X_{\alpha_i}=E_{i,i+1}$ as a root vector of $\alpha_i$ for $\ggg$. Consider $J=J_n$ as before. Then $J=X_{\alpha_1}+X_{\alpha_2}+\cdots+X_{\alpha_{n-1}}$.

\begin{lemma}\label{lem C} Assume  $M=L(\mu)$ is a  polynomial irreducible representation of $G$ with $\mu=\sum_{i=1}^ra_i\epsilon_i$ with $a_i\in\bbn$ and $a_i\geq a_{i+1}$. If $\ucaln=\caln\times M$ has finite $\uG$-orbits, then all $a_i$ are equal,  i.e., $\mu=a_1(\epsilon_1+\cdots+\epsilon_r)$.
\end{lemma}

\begin{proof} We only need to show that if there exists $a_k \in \{a_1,\ldots,a_r\}$ such that $a_k>a_{k+1}>0$, then there are infinite nilpotent $\uG$-orbits in $\ucaln$. In this case, we can take $1\leq k\leq r$ such that $a_1=a_2=\cdots=a_k> a_{k+1}\geq \cdots\geq a_r>0$. According to the assumption, there exists $k+1\leq s<r$ such that $a_s>a_{s+1}$, or $r<n$ and $a_{k+1}=a_{k+2}=\cdots =a_r$.  Let $v_0^-$ be a fixed minimal weight vector of $L(\mu)$. Then $J.v_0^-=\sum_{j=k}^rb_jX_{\alpha_j}v_0^{-}$ with  $b_j=1$ if $a_j>a_{j+1}$ or $0$ if $a_j=a_{j+1}$ for $j=k,\ldots, r$ ($a_{r+1}:=0$). Especially, $b_k=b_r=1$. Fix  $\chi$ or $\chi^{(m)}$ in $\bk$  as previously, and keep the convention (\ref{eq: frobeinius notations}) in mind. Set
$$\textsf{v}_m=(1+X_{\alpha_k}+\texttt{F}_m(\chi)X_{\alpha_r})v_0^-,$$
Then it follows from (\ref{Ad structure}) that those $\textsf{v}_m$ lie in mutually different $\uG$-orbits when $m$ are different. In fact, if $\textsf{v}_{m_1}$ and $\textsf{v}_{m_2}$ lie in the same $\uG$-orbit, then it follows from (\ref{Ad structure}) that there exists $g\in G_{J}$ such that either
$$g.(1+X_{\alpha_k}+\ttF_{m_1}(\chi) X_{\alpha_r})v_0^-\equiv (1+X_{\alpha_k}+\ttF_{m_2}(\chi) X_{\alpha_r})v_0^-\,\text{mod}\;(\jM)$$
  This yields that $m_1=m_2$ {by the same arguments as in the proof of Lemma \ref{lem B}.} %comparing the weight vectors in both sides of the above equality.
\end{proof}

\subsubsection{}
By the above lemmas, we only need to focus on polynomial irreducible modules of the form $M=L(\mu)$ with $\mu=a\sum_{j=1}^r\epsilon_j\in \scrp_n$.
By definition, $M$ is decomposed into the sum of weight spaces: $M=\sum_{\gamma\in X(T)} M_\gamma$ with $M_\gamma=\{v\in M\mid \textbf{t}.v=\gamma(\textbf{t})v\;\forall \textbf{t}\in \textsf{T}\}$. Let $w_0$ be the longest element of the Weyl group $\frak{S}_n$, and fix a minimal weight vector $v_0^-$ of $M$. Then $v_0^-$ is of weight $\mu^{\circ}:=w_0(\mu)=a\sum_{j=n-r+1}^n\epsilon_j$, and $v_0^-$ linearly spans the one-dimensional lowest weight space $M_{\mu^\circ}$ (see \cite[II.2]{Jan1}).

\begin{lemma}\label{lem D} Keep the notations as above, and assume $n>2$. Take $M=L(\mu)$ with $\mu=a\sum_{j=1}^r\epsilon_j\in \scrp_n$, and set $J=J_n$.
\begin{itemize}
\item[(1)] Suppose there are two nonzero weight vectors $w_\eta\in M_{\eta}$ and $w_{\eta'}\in M_{\eta'}$ with different weights
     $\eta=\mu^{\circ}+\alpha_h+\alpha_k$ and $\eta'=\mu^{\circ}+\alpha_{h'}+\alpha_{k'}$ respectively (here $h,h',k,k'\in\{1,2,\cdots,n-1\}$),
     then $\ucaln$ contains infinite $\uG$-orbits.
\item[(2)] If $\ucaln=\caln\times M$ has finite $\uG$-orbits, then one of the following cases occurs: (i) $r=n$; \quad(ii) $r=n-1$ and  $a=1$; \quad(iii) $r=1$ and $a=1$.
% either $r=n-1, n$ or $r=1$ with $a=1$ when $n>2$.
\end{itemize}
\end{lemma}

\begin{proof} (1) Consider $J=J_n$ as before, and take $\chi$ or $\chi^{(m)}$ in $\bk$ as previously.
 Set
$$\textsf{w}_m=v_0^-+\ttF_{m}(\chi)w_\eta+\ttF_{m+1}(\chi)w_{\eta'}.$$
We claim that all $\textsf{w}_m$, $m\in\bbz$, lie in different $\uG$-orbits.  Actually, if $\textsf{w}_{m_1}$ and $\textsf{w}_{m_2}$ lie in the same $\uG$-orbit, there exists $g\in G_{J}$ such that
\begin{align*}\label{eq: about chi}
g.(v_0^-+\ttF_{m_1}(\chi)w_\eta+\ttF_{m_1+1}(\chi)w_{\eta'})
\equiv v_0^-+\ttF_{m_2}(\chi)w_\eta+\ttF_{m_2+1}(\chi)w_{\eta'}\,\,\text{mod}\,(\jM).
\end{align*}
{By the same arguments as in the proof of Lemma \ref{lem B},}
%By comparison of weight vectors on both sides above,
we have $m_1=m_2$.

(2) By the assumption $n>2$, $r$ may change in the range through $1$ to $n$  in the expression $\mu=a\sum_{j=1}^r\epsilon_j$. We divide our arguments into different cases.

(2.i) Suppose $1<r<n-1$. {Then we claim that  $M$ has two non-zero weight spaces $M_{w_0(\mu)+\alpha_{n-r}+\alpha_{n-r-1}}$ containing  nonzero $X_{\alpha_{n-r-1}}X_{\alpha_{n-r}} v_0^-$, and   $M_{w_0(\mu)+\alpha_{n-r}+\alpha_{n-r+1}}$ containing  nonzero vector  $X_{\alpha_{n-r+1}}X_{\alpha_{n-r}} v_0^-$.
%{
Now we show the claim.  Recall that $M$ is an irreducible polynomial representation of $G$, it naturally becomes a $\ggg$-module with a generator $v_0^-$ over $U(\nnn^+)$ because $v_0^-$ is a minimal  weight vector of $M$. Note that $(\mu^\circ, \alpha_{n-r})=-a<0$. According to the weight string theory (see, for example \cite[\S21.3]{hum-0}), $\dim M_{\mu^\circ+\alpha_{n-r}}\ne 0$.  %(see \cite[Ch.33]{hum-1} or \cite[II.2]{Jan1}) where $s_{\alpha_{n-r}}$ denotes %the simple reflection corresponding to the simple root $\alpha_{n-r}$   in the %Weyl group.
Furthermore, $(\mu^\circ+\alpha_{n-r}, \alpha_{n-r-1})=-1$.
By the same reason as above,
$\dim M_{\mu^\circ+\alpha_{n-r}+\alpha_{n-r-1}}\ne 0$.
%=\dim M_{s_{\alpha_{n-r+1}}
%s_{\alpha_{n-r}}(\mu_0)}=1.$$
Hence the corresponding weight space contains the nonzero weight vector   $X_{\alpha_{n-r-1}}X_{\alpha_{n-r}} v_0^-$ because $M$ is $U(\nnn^+)$-generated by the minimal weight vector $v_0^-$.  Hence the claim conerning $X_{\alpha_{n-r-1}}X_{\alpha_{n-r}} v_0^-$ is proved. The claim  concerning $X_{\alpha_{n-r+1}}X_{\alpha_{n-r}} v_0^-$ can be also proved in the same way.

}
 So there are  different nonzero vectors $w_\eta=X_{\alpha_{n-r-1}}X_{\alpha_{n-r}} v_0^-$ and $w_{\eta'}=X_{\alpha_{n-r+1}}X_{\alpha_{n-r}} v_0^-$.
By the statement in part (1), the finiteness of orbits in $\ucaln$ yields  either $r=1$ or $r=n-1$.

(2.ii) Suppose $r=1$ but $a>1$ for $\GL_n$ with $n>2$.  { By the same arguments as in (2.i),} we can show that $M$ has two non-zero weight spaces $M_{w_0(\mu)+2\alpha_{n-1}}$ containing nonzero vector $X_{\alpha_{n-1}}^2 v_0^-$, and   $M_{w_0(\mu)+\alpha_{n-1}+\alpha_{n-2}}$ containing nonzero vector  $X_{\alpha_{n-2}}X_{\alpha_{n-1}} v_0^-$. So we have two different vectors $w_\eta=X_{\alpha_{n-1}}^2 v_0^-$ and $w_{\eta'}=X_{\alpha_{n-2}}X_{\alpha_{n-1}} v_0^-$.
By (1) we have that if $r=1$, then $a=1$ is necessary for the finiteness.

  (2.iii) Suppose $r=n-1$ but $a>1$. { By the same token as above,} $M$ has two non-zero weight spaces $M_{w_0(\mu)+2\alpha_{1}}$ containing a nonzero vector $X_{\alpha_{1}}^2 v_0^-$, and   $M_{w_0(\mu)+\alpha_{1}+\alpha_{2}}$ containing a nonzero vector  $X_{\alpha_{2}}X_{\alpha_{1}} v_0^-$.
By (1), we have that if $r=n-1$, then the finiteness yields $a=1$.

Summing up, we accomplish the proof.
\end{proof}

\begin{theorem}\label{thm: bigg n orbit} Assume $G=\GL_n$ with $n>2$.
Then the nilpotent cone $\ucaln:=\caln(\ugg)$ of the enhanced Lie algebra $\ugg:=\Lie(\uG)$  has finite nilpotent orbits under adjoint $\uG$-action if and only if up to tensor products with one-dimensional modules,  $M$ is isomorphic to the one of the following three modules: (i) one-dimensional module; (ii) the natural module $\bk^n$; (iii) the linear dual module of the natural one $\bk^n$.
\end{theorem}

\begin{proof}
The necessary part follows from Lemmas \ref{lem A}-\ref{lem D}. Next we show the sufficient part. When $M=L(a\sum_{i=1}^n\epsilon_i)$ ($a\in\bbn$), $M$ is a one-dimensional module. It is easily known that $\ucaln$ has $2\aleph_n$ $\uG$-orbits: one containing a single point ``zero vector" $(0, 0)$; one containing all elements $(0, \nu)$ with $\nu$ running over all nonzero elements in $M$, the remaining  ones are of the form $(J_{\lambda}, \nu)$ with $\lambda\in \mathscr{P}_n$ and $\nu$ running over all elements in $M$.
Now we consider $M=L(\epsilon_1)$. In this case, $M$ is the natural module $\bk^n$. The finiteness of orbits in $\ucaln$ follows from the forthcoming Theorem  \ref{thm: classi}. As to Case (iii) when $M$ is the linear dual of the natural module, { by Remark \ref{rem: identical orbti thy},
%$M$ is $G$-conjugate to the natural module. Consequently
the enhanced nilpotent orbits under $\uG$ are in a one-to-one correspondence to the enhanced nilpotent orbits under $\underline{G}^\sharp$ for the natural module.}
So the finiteness in Case (iii)  is assured as in Case (ii). The sufficient part is proved.
\end{proof}

\begin{remark}
%(1) The condition involving complete reducibility in Theorem \ref{thm: %bigg n orbit} comes from some technique requirement when %$\text{ch}(\bk)>0$. It should not be necessary.
%
%(2)
(1) It is meaningful to recall the study of finiteness of nilpotent orbits for a reductive Lie algebra $\ggg=\Lie(G)$ over $\bk$ (see \cite[Theorem 5.8.1]{jan2} and \cite[\S5.11]{Car}). The finiteness of nilpotent orbits of $\ggg$ can be directly deduced from the finiteness of unipotent conjugacy classes of the connected reductive algebraic group $G$ via a Springer bijection when $\text{ch}(\bk)=0$ or $\text{ch}(\bk)=p$ is good for $\ggg$.
The finiteness of unipotent conjugacy classes of a connected reductive algebraic group over an algebraically closed field was finally proved by Lusztig in 1976. This work was fulfilled by great efforts of  Dynkin-Kostant when $\text{ch}(\bk)=0$, Richardson and Lusztig {\sl{etc.}} (see \cite{Kos}, \cite{Ric}, \cite{Lu0}). This finiteness was conjectured by R. Steinberg in 1966.

(2) By Lemma \ref{lem: Fourier transform} (in particular, its second statement) and Remark \ref{rem: identical orbti thy}, there is a close connection between nilpotent orbits in the case (iii) of Theorem \ref{thm: bigg n orbit} and the ones in the case (ii). The present paper will only focus on the latter. As to the classification of nilpotent orbits in the former case, we will have a remark (see Remark \ref{rem: dual natural modules}).
\end{remark}

\subsection{Further discussion for $\GL_2$}\label{sec: 2.4} In the remaining part of this section, we assume $\text{ch}(\bk)=0$.
Suppose $(\rho, M)$ is a given polynomial irreducible  $\GL_2$-representation on $M$ of dimension $m+1$. Then $M$ can be described by a homogeneous polynomials of degree $m$ on $x, y$. This means $M$ can be identified with the space $\Mm$ of homogeneous polynomials on $x,y$ of degree $m$, i.e.
$$\Mm:=\{v=\sum_{j=0}^ma_{j}x^{m-j}y^{j}\mid a_{j}\in \bk\}$$
with the action of
\begin{equation}\label{matrix g}
g=\left( \begin{array}{cc}
 a &b   \cr
c& d\end{array}\right)\in \GL_2
\end{equation}
on $v$ via
$g.v=\sum_{j=0}^m a_{j}{g(x)}^{m-j}{g(y)}^{j}$ where $g(x)=ax+by$ and $g(y)=cx+dy$.
\subsubsection{} Denote by $\bk[\Mm]^{\GL_2}$ the invariants of the affine coordinate ring $\bk[\Mm]$ under $\GL_2$-action.  Consider the algebraic quotient map $\pi: \Mm\rightarrow \Mm/\hskip-4pt/ \GL_2$, i.e.  the morphism corresponding to the inclusion $\bk[\Mm]^{\GL_2}\hookrightarrow \bk[\Mm]$.
Then $\pi$  gives rise to a bijection between the closed orbits of $\Mm$ under $\GL_2$-action and the points of $\Mm/\hskip-4pt/\GL_2$. Let $\caln(\Mm)$ denote the nullcone
$\pi^{-1}(\pi(0))$  which is a closed subvariety of $\Mm$. Set for $i\in\bbz$
$$\Mm(i):=\{v\in \Mm\mid \diag(t,t^{-1})v=t^{-i}v, \forall t\in \bk^\times\}.$$
By definition, $\Mm=\bigoplus_{i=0}^m \Mm(m-2i)$ with
$$\Mm(m-2i)=\bk x^iy^{m-i}.$$

\begin{lemma}\label{lem: gl2 homo dim} Keep the notations and assumptions as above. The following statements hold.
\begin{itemize}
\item[(1)]  $v\in \caln(\Mm)$ if and only if $v$ is $\rho(\GL_2)$-conjugate to a vector from $\bigoplus_{i>0} \Mm(i)$.

\item[(2)] Suppose $m\geq 3$. Then
\begin{align}\label{eq: null dim}
\dim\caln(\Mm)\leq\begin{cases} d+2  &\text{ for }m=2d+1;\cr
d+1 &\text{ for }m=2d.
\end{cases}
\end{align}
\item[(3)] Suppose $m \geq 3$. Then $\dim \Mm\aq\GL_2\geq d$ with $m=2d$ or $m=2d+1$.
\end{itemize}
\end{lemma}
\begin{proof} (1) It can be done  by the same arguments as in the proof of \cite[Lemma 6.3]{GK}.

(2) Set $M^{(m)}_+:=\sum_{i>0}\Mm(i)$, and $M^{(m)}_-:=\sum_{i<0}\Mm(i)$. Clearly, $\dim M^{(m)}_\pm=d$ if $m=2d$, or $d+1$ if $m=2d+1$.
 By (1),  we know $\caln(\Mm)\supset M^{(m)}_+$. Hence $\dim\caln(\Mm)\geq d$.  On the other hand,
 $\rho(\textsf{T})M^{(m)}_-=M^{(m)}_+$ where
 \begin{equation*}
\textsf{T}=\left( \begin{array}{cc}
 0 &1   \cr
1& 0\end{array}\right)\in \GL_2.
\end{equation*}
So $\caln(\Mm)\supset M^{(m)}_-$.

In the following arguments, simply denote $G=\GL_2$ and denote by $B^+$ the standard Borel subgroup consisting of upper triangular matrices of $\GL_2$. Then $\Mmp$ is stabilized by $\rho(B^+)$.
 We can consider $$\widetilde{\caln(\Mm)}:=\{(gB^+, v)\in G\slash B^+\times \caln(\Mm)\mid \rho(g)(v)\in \Mmp\}.$$
Now we take several steps for the remaining arguments by exploiting some classical geometric strategies (see \cite[\S0.15, \S6.20-21]{hum} and \cite[\S6.4]{jan2}).

(i) We claim that $\caln(\Mm)$ is closed in $G/B^+\times \caln(\Mm)$. Actually, we consider the morphism $\phi: G\times \Mm\rightarrow G/B^+ \times \Mm$ sending $(g,v)$ onto $(gB^+,v)$. Then $\phi$ is an open morphism. This is because $\phi$ can be regarded a quotient map
$$G\times \Mm\rightarrow  (G\times \Mm) \slash (B^+\times 0)$$
where $\Mm$ is regarded an algebraic group under addition.
We next have the inverse image $$\phi^{-1}(\widetilde{\caln(\Mm)})=\{(g,v)\in G\times\Mm\mid \rho(g)(v)\in \Mmp\}.$$
On the other hand, the representation $\rho$ gives rise to the natural morphism $G\times \Mm\rightarrow \Mm$ sending $(g, v)$ onto $\rho(g)v$, under which the inverse image of $\Mmp$ exactly coincides with $\phi^{-1}(\widetilde{\caln(\Mm)})$. Hence, $\phi^{-1}(\widetilde{\caln(\Mm)})$ is closed in $G\times \Mm$. Recall that $\phi$ is already known to be open. So the claim is proved.

(ii) We claim that $\widetilde{\caln(\Mm)}$ is irreducible. Actually, consider the morphism $\psi: G\times \Mmp\rightarrow \widetilde{\caln(\Mm)}$ with sending $(g,v)$ onto $(gB^+,\rho(g^{-1})v)$. By definition, $\psi$ is surjective. Consequently, $\widetilde{\caln(\Mm)}$ is irreducible.

(iii) Consider the projection maps $\pi_1:\widetilde{\caln(\Mm)}\rightarrow G\slash B^+$ , and $\pi_2: \widetilde{\caln(\Mm)}\rightarrow \caln(\Mm)$ which are projections onto the first-coordinate, and the second-coordinate,  respectively. Then both of them are $G$-equivariant under the action of $G$ on $\widetilde{\caln(\Mm)}$ via   $h.(gB^+,v):=(hgB^+,\rho(h)v)$. By (1), the morphism
$$ G\times \Mmp\rightarrow \caln(\Mm), \quad (g,v)\mapsto \rho(g)v$$
is surjective. It follows that $\pi_2$ is surjective. On the other hand, for all $g\in G$,
$$\pi_1^{-1}(gB^+)=\{(gB^+,v)\in G\slash B^+\times \Mm\mid \rho(g)(v)\in \caln(\Mmp)\}$$
which is isomorphic to $\rho(g^{-1})\Mmp$. Consequently, $\pi_1^{-1}(gB^+)$ is an irreducible variety of dimension equal to $\dim\Mmp$.

(iv) Focus on $\pi_1$ whose fibers at every points have the same dimension equal to $\dim\Mmp$. Note that both $G\slash B^+$ and $\widetilde{\caln(\Mm)}$ are irreducible. So we have $\dim\widetilde{\caln(\Mm)}=\dim G\slash B^+ +\dim\Mmp=1+\dim\Mmp$ (see \cite[AG, Theorem 10.1]{Bo}).

(v) We finally have
\begin{align*}
\dim \caln(\Mm)\leq \dim \widetilde{\caln(\Mm)}=1+\dim\Mmp=\begin{cases} d+2, &\text{ if }m=2d+1;\cr
d+1, &\text{ if }m=2d,
\end{cases}
\end{align*}
consequently verifying  \eqref{eq: null dim}.

(3) Set $l=\dim \Mm\aq G$. Then the difference $\dim \Mm-\dim\Mm\aq G$ is equal to $(m+1)-l$. Thanks to a theorem owing to  Hochster-Roberts, $\bk[\Mm]^G$ is a Cohen-Macaulay algebra  (see \cite{HR}). Consequently $\Mm\aq G$ is irreducible.  Thus, $\dim\caln(\Mm)=m+1-l$ (see \cite[Theorem 5.1.6]{Sp0} or \cite[AG, Theorem 10.1]{Bo} again). By (2), we have $m+1-l\leq d+1$ for $m=2d$ and  $m+1-l\leq d+2$ for $m=2d+1$. Hence $l\geq d$.

The proof is completed.
\end{proof}

By the above lemma, it is  deduced that $\Mm$ has infinite closed $\GL_2$-orbits whenever $m\geq 3$. Hence, as a direct consequence we have

\begin{corollary} \label{cor: inf}
For $m \geq 3$, there are infinite $\GL_2$-orbits in $\Mm$.
\end{corollary}

\begin{remark} { Lemma \ref{lem: gl2 homo dim} is expected to be valid for  $\GL_n$ with $n$ bigger than  $2$. }

\end{remark}
\subsubsection{}
Recall that for $G=\GL_2$, any polynomial irreducible representation $\rho$ on $M$  is described simply  via a basis,   $\{\textsf{d}\rho(J_2)^{i}v_0^-\mid i=0,1,\ldots, m\}$ where  $v_0^-$ has the same meaning as in \S\ref{sec: general fininteness} and
\begin{equation*}
J_2=\left( \begin{array}{cc}
 0 &1   \cr
0& 0\end{array}\right).
\end{equation*}
 Up to isomorphisms, $M$ is uniquely determined by the dimension, which is denoted by $L(m\epsilon_1)$.

Now we are in a position to state the results for $\GL_2$.

\begin{theorem}\label{thm: gl2 enh class}
 Let $G=\GL(2)$ which is over $\bk$ with $\text{ch}(\bk)=0$, and $(\rho, M)=L(m\epsilon_1)$ ($m\in\bbz_{>0}$) be an irreducible representation on $M$. %rational one-dimensional module.
Then the following statements hold.
\begin{itemize}
\item[(1)] If $m=1$, then there are $4$ enhanced nilpotent orbits.

\item[(2)] If $m=2$, then there are $5$ enhanced nilpotent orbits.

\item[(3)] If $m\geq 3$, then there are infinite enhanced nilpotent orbits.

\end{itemize}
\end{theorem}

\begin{proof} (1) It is clear because $M$ is a natural module. The classification of  orbits can be seen in a general result Theorem \ref{thm: classi} which is forthcoming.

(2) When $m=2$, $\ucaln$ has five $\uG$-orbits: (i) A single point ``zero vector"; (ii) $\textsf{Ad}(\uG)(0, v_0^-)$; (iii) $\textsf{Ad}(\uG) (0, v_0^-+\textsf{d}\rho(J_2)^{2}v_0^-)$; (iv) $\textsf{Ad}(\uG)(J_2, 0)$; (v) $\textsf{Ad}(\uG)(J_2, v_0^-)$. The detailed arguments will be left in the next section (see Lemma \ref{lem: addition m=2}),  where we deal with the closures of nilpotent orbits in this exceptional case.

(3) Suppose $m\geq 3$. Note that $0\times M$ is a $\uG$-stable subvariety of $\ucaln$. For $v_1,v_2\in M$, two vectors $(0,v_1), (0,v_2)\in \ucaln$ lie in different $\uG$-orbits whenever $v_1$ and $v_2$ lie in different $G$-orbits of $M$. Consequently, the desired statements follows from  Corollary \ref{cor: inf}.
\end{proof}

%\begin{proof}

%\end{proof}

\begin{remark} The further details when the situations of finite orbits happen will be seen in the next sections. Also more information on the closures of nilpotent orbits will be presented there.
\end{remark}

\subsection{} As a consequence of the above finiteness theorems, we have the following criteria of the finiteness of $\GL_n$-orbits in $\ucaln$.

\begin{corollary} %Suppose either $\bk$ is of characteristic $0$, or %$\text{ch}(\bk)$ is any but $M$ has complete reducibility.
 Keep the notations as in Theorem \ref{thm: bigg n orbit} and \ref{thm: gl2 enh class}. Then the following statements hold.
\begin{itemize}
\item[(1)] Suppose $n>2$. As a $\GL_n$-variety,  $\ucaln$ has finite orbits if and only if $M$ is isomorphic to one of the following three modules: (i) a one-dimensional module; (ii) The natural module $\bk^n$; (iii) The linear dual module of $\bk^n$.

\item[(2)] When $n=2$ and $\text{ch}(\bk)=0$, as a $\GL_2$-variety, $\ucaln$ has finite orbits if and only if $M$ is an irreducible module of dimension not bigger than $2$.
\end{itemize}
\end{corollary}

\begin{proof}
(1) From an obvious fact that finiteness of $\GL_n$-orbits in $\ucaln$  yields the same property of $\uG$-orbits,  the necessity is a direct consequence of Theorems \ref{thm: bigg n orbit}. Then we verify the sufficiency. As to the case with one-dimensional modules, the sufficiency is clear. Note that the natural module and its dual are $\GL_n$-equivariant.  It suffices to consider the case with the natural modules, which has been verified in \cite{AH}.

For (2), the critical difference from the situation with $\ucaln$ considered as a $\uG$-variety lies in the case when $M$ is an irreducible $3$-dimensional rational module of $\GL_n$. In this case, we claim that there are infinite $\GL_2$-orbits in $\ucaln$.  In order to see this, we consider the $\GL_2$-variety  $\mathcal{X}:=\textsf{Ad}(\uG)(J_2, v_0^-)$. Considering  the centralizer $G_{J_2}$ of $J_2$ in $\GL_2$, we can identify  $G_{J_2}$ with $\bk^\times I_2\times(I_2+\bk J_2)$ for the identity matrix $I_2\in\GL_2$.
 Note that $\textsf{Ad}(G_{J_2}\times M)(J_2,v_0^-)$ is contained in $\{J_2\}\times (\{\bk^\times v_0+\bk v_1+\bk v_2\})$ where $v_i:=J_2^iv_0^{-}$, $i=0,1,2$, forming a basis of $M$. Take $w_c=v_0+cv_2$ for $c\in \bk$. We can show that if  $(J_2,w_c)\in\mathcal{X}$, then these $(J_2,w_c)$ lie in different $\GL_2$-orbits when $c$ takes different values.  Actually,  $(J_2, w_c)$ and $(J_2, w_{c'})$ lie in the same $\GL_2$-orbit if and only if there is a pair $(a,b)\in \bk^\times\times \bk$ such that $av_0+abv_1+acv_2=v_0+c'v_2$, yielding $c=c'$. However $\mathcal{X}$ contains infinite $(J_2,w_{c})$ with $c\in\bk$. Hence $\mathcal{X}$ has infinite $\GL_n$-orbits.  From this observation, the claim follows.

  What remains to do can be carried out by the same arguments  as in the proof for (1). The proof is completed.
\end{proof}

\iffalse
\begin{remark} When $\text{ch}(\bk)=0$, the statements for  two cases in the above corollary can be unified below.
%Consider the enhanced nilpotent cone $\ucaln=\caln(\ggg)\times M$ for any finite-dimensional rational module $M$ of $\GL_n$.
As a $\GL_n$-variety, $\ucaln$ has finite orbits if and only if either $M$ is one-dimensional module, or $M$ is $\GL_n$-equivalent to the natural module $\bk^n$. Note that $M$ can be regarded a $\GL_n$-stable subvariety of $\ucaln$. So the same story happens on $M$:   as a $\GL_n$-variety, $M$ has finite orbits if and only if either $M$ is one-dimensional module, or $M$ is $\GL_n$-equivalent to the natural module $\bk^n$.

\end{remark}
\fi

\section{An exceptional case arising from $\GL_2$ along with its $3$-dimensional representations}\label{sec: exceptional case}

 By Theorem \ref{thm: gl2 enh class}, there are $5$ enhanced nilpotent
 orbits for the enhanced reductive
 group  associated with  $\GL_2$ and its $3$-dimensional irreducible representation $(\rho, L)$ on the space $L:=L(2\epsilon_1)$.
 In this section, we describe their closures
 and the topology relations. We will assume that $\text{ch}(\bk)=0$ in this section.
  %\blue{ Next we will give our explanation about this problem  on the %basis of \S\ref{subsec: on the basis}????}

\subsection{}
 Set $\uG=G\ltimes_\rho L$ for $G=\GL_2$. We have  $\ggg:=\Lie(G)$ equal to $\gl_2$ and  $\Lie(\uG)=\underline{\ggg}:=\ggg\ltimes_{\textsf{d}\rho} L\cong \ggg\oplus L$. Recall that in the sketchy proof of Theorem \ref{thm: gl2 enh class}, there are five $\uG$-orbits in $\ucaln=\caln\times L$. We will give some details.
 \begin{lemma}\label{lem: addition m=2} There are five $\uG$-orbits in $\ucaln$ listed as below.
  (1)   A single point ``zero vector" $(0,0)$, denoted by $\co_1$;
 (2)   $\co_2:=\textsf{Ad}(\uG) (0, v_0^-)$;
 (3) $\co_3:=\textsf{Ad}(\uG) (0, v_0^-+\textsf{d}\rho(J_2)^{2}v_0^-)$;  (4) $\co_4:=\textsf{Ad}(\uG)(J_2, 0)$;  and (5) $\co_5:=\textsf{Ad}(\uG)(J_2, v_0^-)$.
 \end{lemma}

 For the simplicity of notations, we will denote $v_i:=\textsf{d}\rho(J_2)^iv_0^-$, $i=0,1,2$. Then $\{v_i\mid i=0,1,2\}$ is a basis of $L$. Those $5$ orbits can be expressed as:
$$\co_1=\co_{(0,0)};\co_2=\co_{(0,v_0)};
\co_3=\co_{(0,v_0+v_2)};\co_4=\co_{(J_2,0)};\co_5=\co_{(J_2,v_0)}.$$

 \begin{proof} Recall that $L$ can be identified with $M^{(2)}$ in \S\ref{sec: 2.4}, which is the space of homogeneous polynomials on $x,y$ of degree $3$ with $v_0=x^2$, $v_1=xy$ and $v_2=y^2$. Then the $\GL_2$-orbits in $M^{(2)}$  are exactly in a one-to-one correspondence with all standard quadratic forms on $x,y$. On the other hand,
 $$\ucaln=(\{0\}\times L) \cup ((\caln\backslash\{0\})\times L),$$
 both of whose components
  % both $\{0\}\times L\in \ucaln$ and $(\caln\backslash\{0\})\times L$
  are $\uG$-stable with trivial intersection.
 So $$\{0\}\times L=\co_{(0,0)}\cup \co_{(0,v_0)}\cup \co_{(0,v_0+v_2)}$$ is the decomposition into different $\uG$-orbits.

As to the part $(\caln\backslash\{0\})\times L$, by Lemma \ref{lem A}(1) we have that
\begin{align}\label{eq: excep dec}
(\caln\backslash\{0\})\times L=\co_{(J_2,0)}\cup \co_{(J_2,v_0)}.
 \end{align}
 Note that for any $(g,u)\in G_{J_2}\ltimes L$, $\Ad(g,u).(J_2,v_0)=(J_2,-\textsf{d}\rho(J_2)u +\rho(g)v_0)$ with the vector $-\textsf{d}\rho(J_2)u +\rho(g)v_0$ never being zero. This is  because $G_{J_2}\cong \bk^\times I_2\times(I_2+\bk J_2)$ for the identity matrix $I_2\in\GL_2$,
 and then $\rho(g)v_0$ contains a nonzero component vector of weight $-2$ while $\textsf{d}\rho(J_2)u$ never contain any nonzero component of weight $-2$. Hence \eqref{eq: excep dec} is the decomposition into  different $\uG$-orbits. Summing up, we obtain the desired complete decomposition of $\ucaln$ with $\uG$-orbits.
\end{proof}

\subsection{} We have the following main result of this section.
 \begin{prop}\label{prop: 3d prop} Keep the notations and assumptions as above. The following statements hold.
\begin{enumerate}
	\item[(1)] $\dim(\co_1) =0$; Naturally $\overline{\mathcal{O}_1}=
\mathcal{O}_1$.
	
	\item[(2)] $\dim(\co_2)
=3$; and $\overline{\mathcal{O}_2}=\mathcal{O}_1\cup
\mathcal{O}_2$.
	\item[(3)] $\dim(\co_3)
=3$;    and $\overline{\mathcal{O}_3}=\{0\}\times L=
\mathcal{O}_1\cup \mathcal{O}_2\cup\mathcal{O}_3$.
	\item[(4)] $\dim(\co_4)=4$; and $\overline{\mathcal{O}_4}=
\mathcal{O}_1\cup \mathcal{O}_4$.
	\item[(5)] $\dim(\co_5)
=5$; and $\overline{\mathcal{O}_5}=\underline{\mathcal{N}}$.
\end{enumerate}
	\end{prop}
 The proof will be given in the next subsections.

 \subsection{} We first adopt some general arguments  on the closures of nilpotent orbits in the reductive group case.  Suppose $P$ is a parabolic subgroup of $G$, and denote  $\underline{P}=P\ltimes_{\rho} L$ (here and further $\rho$ is actually $\rho|_P$. We do not distinguish them in the notations. Sometimes, we even use the notation $P\ltimes L$ for $\underline{P}$) (see \cite[\S8]{jan2}).
 Suppose $M$ is a $\underline{P}$-stable
 subspace of $\ugg$. Set
 \begin{align}\label{equ:4.4}
 \mathcal{P}_M:=\{(\ug\underline{P},(X,w))\in
 \underline{G}/\underline{P}\times \ugg \mid
 \ug\in \uG \text{ with }
 \Ad(\ug)^{-1}(X,w) \in M\}.
 \end{align}
 Then we have  $p_1:\mathcal{P}_M\rightarrow \uG\slash \underline{P}$ and $p_2: \mathcal{P}_M\rightarrow \ugg$, which are defined as the first-coordinator projection and the second-coordinator projection, respectively.
It is readily known that  two projections $p_1$ and $p_2$ are $\underline{G}$-equivariant.

 \begin{lemma}\label{lem:6.3} The
 following statements hold.
 	\begin{enumerate}
 		\item [(1)] The subset $\mathcal{P}_M$ of
 $\uG/\underline{P} \times \ugg$ is a closed
 subvariety.
 		It is smooth, irreducible and  of dimension
 $\dim(G/P) +\dim(M)$.
 		
 		\item [(2)] The first-coordinator projection $p_1 : \mathcal{P}_M \rightarrow
 \underline{G}/\underline{P}$ is a vector bundle of rank $\dim(M)$.
 		
 		\item [(3)] The subset $\textsf{Ad}(\uG) M$ of $\ugg$ is closed.
 		
 		\item [(4)] For all $(Y,u) \in M$
 		$$\dim(\textsf{Ad}(\uG) (Y,u)) \leq
 \dim(\underline{G}/\underline{P}) +\dim(\textsf{Ad}(\underline{P})(Y,u))$$
 		with equality if and only if
 $(\underline{G}_{(Y,u)})^{\circ}=\underline{G}_{(Y,u)}\subset
 \underline{P}$ where $(-)^\circ$ stands for the identity connected component of an algebraic group $(-)$.
 		 	\end{enumerate}
 	%%%%%%%%%%%%%%%%%%%%%%%%%%%%%%%%%%%%%%%%%%%%%%%%%%%%%%%
 	
 \end{lemma}

 \begin{proof}
 	Let $\varphi:\underline{G} /\underline{P}\simeq G/P$
 	send  $\underline{g}\underline{P}$ onto $gP\in G\slash P$. Then  $\varphi$ is an isomorphism of this two
 varieties. 	With this isomorphism, the arguments in the proof of  \cite[Lemma 8.7]{jan2} can be carried out here. 		 \end{proof}

\subsection{} Keeping the set-up above, we additionally assume that there exists $(X,w) \in M$ such that
\begin{align}\label{equ:4.8}
\overline{\textsf{Ad}(\underline{P})(X,w)}=M \quad and \quad
(\uG_{(X,w)})^{\circ}\subset \underline{P}.
\end{align}
Let $\mathcal{O}_{(X,w)}=\textsf{Ad}(\uG) (X,w)$. Then \eqref{equ:4.8} and Lemma \ref{lem:6.3} imply that
$$\dim(\mathcal{O}_{(X,w)})=\dim(\underline{G}/\underline{P})+\dim M.$$
%Recall $\ugg_{(X,w)}$ is  the centralizer of $(X,w)$ in $\ugg$, which is %equal to $\{(Y,u)\in\ggg\mid Yw-Xu=0,
%[Y,X]\,=\,0\}$.

\begin{lemma}\label{lem:4.5}
	Under the above assumptions, we have that
 $\overline{\mathcal{O}_{(X,w)}}=\textsf{Ad}(\uG)M$.
\end{lemma}

\begin{proof}
By definition, $(X,w)\in M$ implies that $\mathcal{O}_{(X,w)}\subset
\textsf{Ad}(\uG) M$. Note that $\textsf{Ad}(\uG) M$ is closed by Lemma
\ref{lem:6.3}(3). So $\overline{\mathcal{O}_{(X,w)}}\subset
\textsf{Ad}(\uG)M$. On the other hand,  $M\subset
\overline{\mathcal{O}_{(X,w)}}$	by \eqref{equ:4.8}, hence $\textsf{Ad}(\uG)M \subset \overline{\mathcal{O}_{(X,w)}}$. So
$\overline{\mathcal{O}_{(X,w)}}=\textsf{Ad}(\uG) M$.
\end{proof}

Proposition \ref{prop: 3d prop} is a consequence of some applications of the above lemma.
\subsection{Proof of Proposition \ref{prop: 3d prop}}
 We will take some steps for the proof.

(i) Note that
$\dim(\co_{(X,w)})=\dim(\underline{\ggg})-\dim(\underline{\ggg}_{(X,w)})$.
It suffices to compute each $\underline{\ggg}_{(X,w)}$ for these five
orbits.
By a direct computation, the numbers of $\dim(\underline{\ggg}_{(X,w)})$ for these five orbits are $7,5,4,3,2$,  respectively. So $\dim(\co_1) =0$ and
$\dim(\co_i) =i$ $(i=2,3,4,5)$. Thus, the parts of all statements on dimensions are proved.  What remains to do is to verify the statements involving closures $\overline{\co_i}$ except the trivial case $\overline{\co_1}$ of the single point.

(ii)  To describe the orbit closures, we only need to choose suitable
$\underline{P}_i$ and $M_i$ for each $\co_i$ satisfying the assumptions  for Lemma \ref{lem:4.5}.
We choose these $\underline{P}_i$ and $M_i$ for each case as below.
\begin{itemize}
%\item[Case $\co_1$]:
%$\underline{P}_1=\underline{G}, M_1=\{0\}\times \{0\}$ with %$(X,w)=(0,0)$
\item[Case $\co_2$]: $\underline{P}_2=B^{-}\ltimes L,M_2=\{0\}\times \{\mathbb{C}v_0\}$ with $(X,w)=(0,v_0)$. Here and further, $B^\pm$ denotes the subgroup of $G$ consisting of upper/lower triangular matrices respectively.

\item[Case $\co_3$]: $\underline{P}_2=\underline{G},M_3=\{0\}\times L$ with $(X,w)=(0, v_0+v_2)$;

\item[Case $\co_4$]: $\underline{P}_4=B^{+}\ltimes L,M_4=\mathfrak{u}\times \im J_2$ with $(X,w)=(J_2,0)$. Here and further $\mathfrak{u}$ denotes  the Lie algebra of all strictly upper triangular matrices, i.e. $\uuu=\bbk J_2$, and $\im J_2$ denotes the image subspace  of $\textsf{d}(\rho)(J_2)$ in $L$.

\item[Case $\co_5$]: $\underline{P}_5=B^{+}\ltimes L,M_5=\mathfrak{u}\times L$ with $(X,w)=(J_2,v_0)$.
\end{itemize}
 For the cases $\co_2$, $\co_3$ and $\co_5$, it is easily checked that the assumption required in Lemma \ref{lem:4.5} are satisfied. Hence $\overline{\co_i}=\textsf{Ad}(\uG)M_i$ for $i=2,3, 5$. Especially, $\overline{\co_2}$ is a proper nontrivial closed subset  of ${0}\times L=\overline{\co_3}$. And $\overline{\co_5}=\caln\times L$. The corresponding statements are proved.

(iii) What it is worthwhile to talk more about is  the case $\co_4$.
Let us first explain that the assumption in Lemma \ref{lem:6.3} satisfies for the case $\co_4$. By computation, the $\underline{P_4}$-orbit $\Ad(\underline{P_4})(J_2, 0)=\bk^{\times}J_2\times (\bk v_1+\bk v_2)$, which is dense in $M_4$.  Furthermore, the centralizer $\uG_{(J_2,0)}$ is equal to the subgroup $B^+\times \bk v_2$ of $P_4$. So \eqref{equ:4.8} also satisfies. Lemma \ref{lem:6.3} is valid to this case, which implies $\overline{\co_4}=\textsf{Ad}(\uG)M_4$.

Next, we will show that $\overline{\co_4}=\mathcal{O}_1\cup \mathcal{O}_4$.
Obviously, $\mathcal{O}_1\cup \mathcal{O}_4 \subseteq
\overline{\mathcal{O}_4}$. From the construction, it is easily deduced that $\overline{\mathcal{O}_4}=\mathcal{O}_1\cup \mathcal{O}_4$ if and only if $\overline{\co_2}$ is not contained in the closure $\overline{\co_4}$. In this case, $\overline{\co_3}$ is consequently not contained in $\overline{\co_4}$.
We will show this by reductio ad absurdum below.

 Suppose  $\overline{\co_2}\subset \overline{\co_4}$.
Then  $(0,v_0)\in \underline{G}\cdot M_4$.
 Recall $M_4=\uuu\times \im J_2=\uuu\times (\bk v_1+\bk v_2)$.
Thus there must be some $\ug=(g, u)\in \uG$ such that
$\Ad(\ug)(0,v_0)\subset \uuu\times \im J_2$, which implies  $\rho(g)(v_0)\in \bk v_1+\bk v_2$, contradicting the representation property of $\rho$ because the weight of $v_0$ is $-1$ while there is no vector of weight $-1$ in $\bk v_1+\bk v_2$.
\iffalse
 $=\Ad(\sigma,u)(0,k_1v_1+k_2v_2)=\sigma(k_1v_1+k_2v_2)$ for some
$(\sigma,u)\in \underline{G}$, where $k_1,k_2$ are not all zero. We may
assume that $\sigma=\left(\begin{matrix}
\sigma_{11} & \sigma_{12}\\
\sigma_{21} & \sigma_{22}
\end{matrix}\right)$,
then $\sigma
v_1=\sigma_{11}\sigma_{12}v_2+(\sigma_{11}\sigma_{22}+\sigma_{21}\sigma_{12})v_1+\sigma_{21}\sigma_{22}v_0$
and  $\sigma
v_2=\sigma_{11}^2v_2+2\sigma_{11}\sigma_{21}v_1+\sigma_{21}^2v_0$.
Thus we get the following equations by $v_0+v_2=\sigma(k_1v_1+k_2v_2)$
%\begin{equation}
%\begin{cases}
%$\sigma_{11}(k_1\sigma_{12}+k_2\sigma_{11})=1$\\
%$\sigma_{11}(k_1\sigma_{22}+k_2\sigma_{21})+\sigma_{21}(k_1\sigma_{12}+k_2\sigma_{11})=0$\\
%%$\sigma_{21}(k_1\sigma_{22}+k_2\sigma_{21})=1$
%\end{cases}
%\end{equation}
\begin{equation*}
\left\{
	\begin{aligned}
	&\sigma_{11}(k_1\sigma_{12}+k_2\sigma_{11})=1\\
&\sigma_{11}(k_1\sigma_{22}+
k_2\sigma_{21})+\sigma_{21}(k_1\sigma_{12}+k_2\sigma_{11})=0\\
	&\sigma_{21}(k_1\sigma_{22}+k_2\sigma_{21})=1
		\end{aligned}
		\right.
		\end{equation*}
\fi
Hence, $\mathcal{O}_2\nsubseteq \overline{\mathcal{O}_4}$. So we have $\overline{\mathcal{O}_4}= \mathcal{O}_1\cup
\mathcal{O}_4$.

The proof is completed.

%\subsection*{Acknowledgement}  We are thankful to Jun Lu for his helpful %discussion.

%\end{proof}

\section{Enhanced nilpotent orbits under enhanced group action associated with natural modules}
In this section, we set $\uG=G\ltimes_\eta V$ for $G=\GL(V)$, associated with the natural representation $\eta$ of $G$ on $V\cong \bk^n$. All vector spaces are assumed over an algebraically closed field $\bk$ of any characteristic. By abuse of notations, we sometimes use the notation $\GL_n$ for $\GL(V)$ (resp. $\gl_n$ for $\gl(V)$) alternatively.

\subsection{}\label{subsec: enh conj notations} As $\GL_n$ is a closed subgroup of $\uG$, any nilpotent $(X,w)\in \ucaln$ is preliminarily conjugate to a certain $(J_\lambda,v)$ under the action of $\GL_n$ where $J_\lambda$ is the Jordan standard matrix of $X$ with $\lambda=(a_1,\ldots,a_t)\in\scrp_n$. If $X=0$, then naturally $J_\lambda=0$ and then all $(0,w)$ are in the same $\uG$-orbit for nonzero vectors $w\in V$.

Suppose $X$ is nonzero (note that the following arguments becomes trivial but still valid when it is the case $X=0$), then we can write $\lambda=(a_1\geq a_2\geq \cdots\geq a_t>0)$. Furthermore, $\lambda$ can be further written as
\begin{align}\label{new lam express}
(b_1^{d_1}b_2^{d_2}\ldots b_r^{d_r}) \text{ with }b_1>b_2>\cdots>b_r,
 \end{align}
 which means that each different $b_j$ occurs $d_j$ times with all $d_j>0$, $\sum_{j=1}^r d_j=t$, and $\sum_{j=1}^r d_jb_j=n$.
Recall that associated with $X$ there exist a basis of $V$ as below:
\begin{align}\label{eq: Jordan st basis 1}
&v_1,Xv_1,\ldots, X^{a_1-1}v_1;\cr
&v_2,Xv_2,\ldots, X^{a_2-1}v_2;\cr
&\ldots\ldots\ldots\cr
&v_t,Xv_t,\ldots, X^{a_t-1}v_t.
\end{align}
Such a basis in \eqref{eq: Jordan st basis 1} will be called a {{\sl Jordan basis}} associated with $X$ (of type $\lambda\in \calp_n$).
According to the expression of $\lambda$ in (\ref{new lam express}), the above basis can be reformulated as below:
\begin{align}\label{eq: Jordan st basis 2}
&v_1,Xv_1,\ldots, X^{b_1-1}v_1;\cr
&\ldots\ldots\ldots\cr
&v_{d_1},Xv_{d_1},\ldots, X^{b_1-1}v_{d_1};\cr
&v_{d_1+1},Xv_{d_1+1},\ldots, X^{b_2-1}v_{d_1+1};\cr
&\ldots\ldots\ldots\cr
&v_{d_1+d_2},Xv_{d_1+d_2},\ldots, X^{b_2-1}v_{d_1+d_2};\cr
%&v_{d_1+d_2+1},Jv_{d_1+d_2+1},\ldots, J^{b_3-1}v_{d_1+d_2+1};\cr
&\ldots\ldots\ldots\cr
&\ldots\ldots\ldots\cr
&v_{d_1+d_2+\cdots+d_{r-1}+1},Xv_{d_1+d_2+\cdots+d_{r-1}+1},\ldots, X^{b_r-1}v_{d_1+d_2+\cdots+d_{r-1}+1};\cr
&\ldots\ldots\ldots\cr
&v_{d_1+d_2+\cdots+d_r}\;(=v_t),Xv_{t},\ldots, X^{b_r-1}v_{t}.
\end{align}

This basis  can be endowed with a grading: $\gr(X^kv_i)=-(a_i-1)+2k$ with $k=0,1,\ldots,a_i-1$, and $i=1,\ldots,t$. Then $V$ is endowed with a $\bbz$-graded structure
$$V=\sum_{m\in \bbz}V_m.$$
It is easily verified that $\dim V_m=\dim V_{-m}$.
\iffalse
Naturally, $\gl_n$ is endowed with a $\bbz$-structure with $\gl_n=\sum_{k\in\bbz}\gl_n^{(k)}$ where
$$\gl_n^{(k)}=\{Y\in \gl_n\mid Y(V_m)\subset V_{m+k} \;\forall m\in\bbz\}.$$
Set $\gl_n^+=\sum_{k>0}\gl_n^{(k)}$, and $\gl_n^{\geq m}=\sum_{k\geq m}\gl_n^{(k)}$.
\fi

Now we introduce a gradation in $\ggg$ associated with $X$.
For any $t\in \bk^\times$, let $\tau(t)\in \GL(V)$ be the linear map such that $\tau(t)v=t^mv$ for  $v\in V_m, \forall m\in \bbz$. Then $\tau$ is a cocharacter of $G=\GL(V)$.
We can endow $\gl_n$ with a $\bbz$-gradation via $\gl_n=\sum_{k\in\bbz}\gl_n^{(k)}$ where
$$\gl_n^{(k)}=\{Y\in \gl_n\mid \textsf{Ad}(\tau(t))(Y)=t^k Y \text{ for all } t\in \bk^\times\}.$$
Set $\gl_n^+=\sum_{k>0}\gl_n^{(k)}$, and $\gl_n^{\geq m}=\sum_{k\geq m}\gl_n^{(k)}$.

By \cite[\S5]{jan2}, $\frak{p}:=\sum_{k\geq 0}\gl_n^{(k)}$ is a parabolic subalgebra of $\gl_n$, which is the Lie algebra of the parabolic subgroup $P$ of $\GL_n$ associated with $\tau$. This $P$ can be described as the set of all $g\in G$ such that $\lim_{t\rightarrow 0}\tau(t)g\tau(t)^{-1}$ exist (see \cite[\S8.4]{Sp0}). The gradation of $\ggg$ is compatible with the one of $V$. In particular, associated with the grading of $V=\sum_{m\in\bbz}V_m$, one has the filtration $\{V_{\geq m}=\sum_{j\geq m}V_j\mid m\in\bbz\}$ which is preserved by $P$-action.
 %defined via $P=\{g\in \GL_n\mid g(V_{\geq_m})=V_{\geq m}\text{ for all %}m\in\bbz\}$. Here $V_{\geq m}=\sum_{k\geq m}V_k$.

\subsection{} For the simplicity of statements, we simply consider the Jordan standard matrix  $J_\lambda$ of $X$ without loss of any generality.

\begin{lemma}\label{lem: quotient orb}
Suppose $J=J_\lambda$ is a nonzero Jordan standard  matrix {associated} with $\lambda$. Set $\widetilde V=V\slash \im J$, denote by $G_J$  the centralizer of $J$ in $G=\GL_n$. Then the following statements hold.
\begin{itemize}
\item[(1)] $\widetilde V$ is stabilized by $G_J$, and there are $(r+1)$ $G_J$-conjugacy classes in $\widetilde V$ with representatives respectively: $\bar u_j:=u_j+\im J$ for $j=1,\ldots,r$ for $u_j:=v_{d_1+\cdots+d_j}$, along with setting $\bar u_{r+1}:=0 $ (also  $u_{r+1}:=0$).
    \item[(2)] Let $q=d_0+d_1+\cdots+d_{j-1}$ ($j=1,\ldots,r+1$, $d_0$ and $d_{r+1}$ is indicated zero). Then $$\ggg_J.u_j=\sum_{k>q}\sum_{0\leq i_k\leq a_k-1}\bk J^{i_k} v_k.$$
         \end{itemize}

\end{lemma}

\begin{proof} (1) Denote by $\GL_{d_j}$ ($j=1,\ldots,r$) the general linear group of the subspace of $V$ spanned by
%\begin{align}\label{space}
$${v}_{d_1+\cdots+d_{j-1}+1},\ldots,{v}_{d_1+\cdots+d_{j-1}+d_j}.$$
%\begin{align}
Then we have a reductive group $C:=\GL_{d_1}\times \cdots\times\GL_{d_r}$. And  all nonzero vectors in the subspace spanned by these vectors fall in the same $GL_{d_j}$-orbit.
Furthermore, $G_J=C\ltimes R$ where $R$ is a group generated by the unipotent elements of $\GL_n$ of the form $\sigma=\id+\sigma^+$ such that $\sigma^+\in \gl_n^+\cap \ggg_{J}$  (see \cite[\S3]{jan2}).

 By definition, $\widetilde V$ is stabilized by $G_J$, which is correspondingly spanned by a basis $\{\bar v_i\mid i=1,\ldots,t\}$.
For any given nonzero vector $\bar v$ of $\widetilde V$, $\bar v$ is linearly spanned by the basis $\{\bar v_i\mid i=1,\ldots,t\}$:
$$\bar v=\sum_{i\geq k} c_i\bar v_i$$
where $c_{k}\ne 0$. By the above arguments, up to $C$-conjugation we might as well suppose $k=d_1+\cdots+d_j$ for a certain $j\in\{1,\ldots,r\}$ and $c_k=1$. By the same reason, we can further write  $\bar v=\bar v_{d_1+\cdots+d_j}+\sum_{q=j+1}^r c_{d_1+\cdots+d_q} \bar{v}_{d_1+\cdots+d_q}$.

Note that all unipotent elements from $R$ preserve the filtration of $V$. As long as $c_{d_1+\cdots+d_{j+1}}$ is nonzero, we can choose a suitable $\sigma_{j+1}=\id+\sigma_{j+1}^+\in R$ with $\sigma_{j+1}^+\in \gl_n^{\geq(b_{j}-b_{j+1})}$ such that $\sigma_{j+1}(\bar v)=\bar v_{d_1+\cdots+d_j}+\sum_{q=j+2}^r * \bar{v}_{d_1+\cdots+d_q}$.
Actually, consider $Z\in \gl_n^{\geq(b_{j}-b_{j+1})}$ which is defined via
\begin{align}
Z(v_i)=
\begin{cases} -c_{d_1+\cdots+d_{j+1}}v_{d_1+\cdots+d_{j+1}},  &\text{ if } i=d_1+\cdots+d_j;\cr
0, &\text{ otherwise}
\end{cases}
\text{ for } i=1,\ldots, t.
\end{align}
By \cite[\S3.1]{jan2}, $Z$ is fully defined as a linear transformation of $V$ commutating with $J$. Furthermore, $Z^2=0$. Hence $\sigma_{j+1}=\id+Z$ is required.

Recursively, we can suitably choose a sequence $\sigma_q=\id+\sigma_q^+\in R$ $(q=j+2,\ldots,r)$ with $\sigma_{q}^+\in \gl_n^{\geq (b_{j}-b_{q})}$ such that
 we finally have $\sigma(\bar v)=\bar v_{d_1+\cdots+d_{j}}$ for  $\sigma=\sigma_{r}\circ\sigma_{r+1}\circ\cdots\circ\sigma_{j+1}$. The proof is completed.

 (2) By Lemma \ref{lem: basic lem}, $\ggg_J=\Lie(G_J)$. By the above arguments, $\ggg_J.u_j$ exactly coincides with $\sum_{k>q}\sum_{0\leq i_k\leq a_k-1}\bk J^{i_k} v_k$. The statement follows.
\end{proof}

For any $v\in V$ we have $e^v:=(e,v)\in \uG$ (similarly $e^V:=\{e\}\times V$). Then $v\mapsto e^v$ gives rise to an algebraic group monomorphism from $V$ to the unipotent radical $\text{R}_{\text u}(\uG)=e^V$ of $\uG$. So we identify $V$ with $\text{R}_{\text u}(\uG)$. With aid of the adjoint $e^V$-action,  we have the following consequence of Lemma \ref{lem: quotient orb}.

\begin{prop}\label{thm: classi} As a $\uG$-variety, any nonzero element $(X, w)$ in $\ucaln$ lies in a unique orbit represented as below
\begin{itemize}
    \item[(1)] if the nonzero $X\in \caln(\ggg)$ has a Jordan standard matrix $J_\lambda$ with $\lambda=(a_1\geq a_2\geq\cdots\geq a_t>0)$, then $(X,w)$ lies in the orbit represented by $(J_\lambda, 0)$ if $w\in \im(J)$ (in this case we will say that $(X,w)$ is of type $\lambda[d_1+\ldots+d_r]$),
        or otherwise,  lies in a unique one from the orbits represented by $(J_\lambda, u_j)$ with $u_j=v_{d_1+\cdots+d_j}$ as in \eqref{eq: Jordan st basis 2}, $j=1,2,\cdots,r$. In the latter case, we say that the corresponding orbit is of type $\lambda[q]$ for $q=d_0+d_1+\cdots+d_{j-1}$ with $d_0:=0$.
     \item[(2)] If $X=0$, then all nonzero $(0,w)$ lies in the same $\uG$-orbit for all $w\in V\backslash \{0\}$. In the same sense as in Statement (1), the corresponding orbit is said to be of  type $\lambda[0]$ for $\lambda=(1^n)$.
         \item[(3)] The point $(X,w)=(0,0)$ itself forms a single-point orbit which is said to be of type $\lambda[n]$ for $\lambda=(1^n)$.
        \item[(4)] For different  $\lambda$, $\mu\in\scrp_n$, any two $(J_\lambda, v_\lambda)$ and $(J_\mu,v_\mu)\in \ucaln$ are not in the same $\uG$-orbit.
        \end{itemize}
\end{prop}
\begin{proof} By definition, the last three statements  are clear. We then prove the first one. In the following, we suppose $X\ne 0$, and $X$ is conjugate to $J_\lambda$ by $\GL_n$-action.
 Note that $\GL_n$ is naturally a closed subgroup of $\uG$ by identifying $\GL_n$ with $GL_n\times\{0\}\subset \uG$.
So, for any given $(X,w)\in \ucaln$, $(X,w)$ is conjugate to $(J, v)$ for $J=J_\lambda$ with $J=\Ad(g)(X)$ and $v=g.w$ for $g\in \GL_n$.

Now we focus on $(J,v)\in \ucaln$. We first have $\Ad(e^u)(J,v)=(J, v-J.u)$ for any $u\in V$. When $u$ runs through $V$, we immediately have
that the $\uG$-orbits are determined by the $G_J$-conjugacy classes of $\widetilde V$. If $v\in \im(J)$, then $(J,v)$ falls in the $\uG$-orbit where $(J,0)$ lies.  Suppose $v\notin\im(J)$. By Lemma \ref{lem: quotient orb}, we have that $(J,v)$ lies in a certain $\uG$-orbit where $(J,u_j)$ lies.

What remains to do is to verify  that $(J,u_j)\in \ucaln$, $j=1,\ldots,r,r+1$ lie in different $\uG$-orbits (note that $u_{r+1}$ is already indicated zero). It is sufficient to show different $u_j$ are not conjugate by $G_J$-action. This is assured by the structure of $G_J$.  The proof is completed.
\end{proof}

\begin{example} Keep in mind that for $X\in \caln$ which has a Jordan basis in \eqref{eq: Jordan st basis 1} (admitting a Jordan standard matrix $J_\lambda$ for $\lambda\in \scrp_n$), $(X,0)\in \ucaln$ is of type $\lambda[d_1+d_2+\cdots+d_r]$, and $(X,v_1)$ (even any of  $(X,v_i)$, $i=1,2,\ldots,d_1$)  is of type $\lambda[0]$. In particular, when $X=0$, the corresponding partition is $\lambda=(1^n)$. In this case $(X,w)$ is of type $\lambda[0]$ if $w\ne 0$, and of type $\lambda[n]$ if $w=0$.
\end{example}

\begin{conven} (1) From now on, we let $\co_\lambda$ denote  the $G$-orbit of the nilpotent element $X$ of Jordan standard type corresponding to $\lambda$, and $\co_{\lambda[q]}$ denote the $\uG$-enhanced nilpotent orbit of $(X,w)$ assumed of type $\lambda[q]$. Sometimes, we also directly write $\uG.(X,w)$ (resp. $G.X$) for the  orbit of $(X,w)$ under adjoint $\uG$-action (resp. the orbit of $X$ under adjoint $G$-action).

(2) Denote by $\scrpe$ the set of all $\lambda[q]$ with $\lambda=(b_1^{d_1}\ldots b_r^{d_r})$ ranging over $\scrp_n$ and $q$ ranging over $\{\sum_{k=0}^{j-1}d_k\mid j=1,\ldots,r+1\}$. Elements in $\scrpe$ may be called enhanced partitions of $n$.

\end{conven}

Then we can give a concise  classification of enhanced nilpotent orbits.

\begin{theorem}\label{prop: altn nilp cls} Keep the notations and assumptions as before. Then $\{\co_{\lambda[q]}\mid \lambda[q]\in \scrpe\}$ is  the complete set of enhanced nilpotent orbits under adjoint $\uG$-action, which are parameterized by enhanced partitions of $n$, equivalently to say,  by the set $\{(\lambda, q)\mid \lambda=(b_1^{d_1}\ldots b_r^{d_r})\in \scrp_n, \text{ and } q=d_0+d_1+\cdots+d_{j-1}, j=1,\cdots,r,r+1\}$.
\end{theorem}

\begin{remark}\label{rem: dual natural modules}
With the above theorem (or by the same arguments again), one can get the same parameters of nilpotent orbits for  $\underline{\uG}=G\ltimes_\zeta V^*$ for $G=\GL(V)$, associated with the natural representation $\zeta$ of $G$ on $V^*$.
\end{remark}

\subsection{The dimensions of $\co_{\lambda[q]}$} For an arbitrarily given $(X,w)\in \ucaln$ of type $\lambda[q]$, we will compute the dimension of $\co_{\lambda[q]}=\textsf{Ad}(\uG)(X,w)$. For this, consider the stabilizers $\uG_{(X,w)}$ of $(X,w)$ in $\uG$, and the centralizers $\ugg_{(X,w)}$ of $(X,w)$ in $\ugg$. Similarly, we have $G_X$ and $\ggg_X$.
Note that $\Lie(G_X)=\ggg_X$, and $\Lie(\uG_{(X,w)})=\ugg_{(X,w)}$ (Lemma \ref{lem: basic lem}).
 It suffices to compute $\dim \ugg_{(X,w)}$.
Consider the subspace $\im X$ of $V$ associated with  the linear map $X:V\rightarrow V$, and denote $M=\im X+\ggg_X.w$ which is still a subspace of $V$. We define a linear map
$$\Upsilon: \ggg_X\times V\rightarrow M, \; \; (Y,u)\mapsto -Xu+Yw. $$
Obviously, $\Upsilon$ is surjective. And $\ugg_{(X,w)}$ is actually the kernel of $\Upsilon$.  So $\dim\ugg_{(X,w)}$ is equal to the difference of $\dim (\ggg_X\times V)$ and $\dim M$. We first have the following fact.
\begin{lemma}\label{lem: about M} Let $M=\im X+\ggg_X.w$. Then $\dim M=n-q$.
\end{lemma}
\begin{proof} When $X=0$, the statement is clear. Suppose $X\ne 0$. By  Lemma \ref{lem: quotient orb}, we have $M=\im X+\sum_{k=q+1}^t\bk v_k$. Hence $\dim M=n-q$.
\end{proof}

\begin{prop} \label{prop: enh orbit dim}
Keep the notations as above. Then
$$\dim \co_{\lambda[q]}=\dim\co_\lambda+(n-q).$$
\end{prop}

\begin{proof} By the above arguments, $\dim \co_{\lambda[q]}=\dim \ugg-\dim \ggg_{(X,w)}$. However, $\dim \ggg_{(X,w)}=\dim (\ggg_X\times V)-\dim M=\dim\ggg_X+q$. So we have
$ \dim \co_{\lambda[q]}=n^2+n- (\dim\ggg_X+q)=\dim\co_\lambda+(n-q)$. The proof is completed.
\end{proof}

\begin{corollary}\label{CorBasicClosure} The following statements hold.
\begin{itemize}
\item[(1)] $\overline{\co_{\lambda[q]}}\subset \overline{\co_{\lambda}} \times V$.

\item[(2)] $\overline{\co_{\lambda[0]}}=\overline{\co_{\lambda}}\times V$

\item[(3)] $\dim(\overline{\co_{\lambda[d_1+\cdots+d_r]}})=
    \dim(\overline{\co_\lambda})+\dim(\im J_\lambda)$ for $\lambda=(b_1^{d_1}\ldots b_r^{d_r})\in \scrp_n$.
\end{itemize}
\end{corollary}
\begin{proof} (1) is clear. As to (2), both $\overline{\co_{\lambda[0]}}$ and $\overline{\co_{\lambda}}\times V$ are irreducible closed $\uG$-invariant subset. By Proposition \ref{prop: enh orbit dim}, both have the same dimension. So both have to coincide. The second statement follows.
As to (3), it is a direct consequence of Proposition \ref{prop: enh orbit dim}.
%$\co_{\lambda[d_1+\cdots+d_r]}=\uG. (X,0)$ with $X$ admitting a %standard Jordan form of type $\lambda$. In particular, %$\overline{\co_\lambda}\times \im X$ is an irreducible $\uG$-invariant %closed subset of $\ucaln$. By Proposition \ref{prop: enh orbit dim}, %$\overline{\co_\lambda}\times \im X$ has the same dimension as %$\overline{\co_{\lambda[d_1+\cdots+d_r]}}$. So both have to coincide. %The third statement follows.
\end{proof}

\subsection{Enhanced numbers for enhanced nilpotent orbits}
%We now give the definition of the partial order on $G$-orbits
%in $\caln\times %V=\bigcup\limitsP_{(\mu;\nu)\in\callp_n}\co_{\lambda-(1)_k,(1)_k}$ %firstly.

For the study of closures of nilpotent orbits later, we first introduce enhanced numbers for enhanced nilpotent orbits.

Keep the notations as in  \eqref{new lam express}, \eqref{eq: Jordan st basis 1} and \eqref{eq: Jordan st basis 2}. We make an additional appointments $d_0=0$ and $d_{r+1}=0$. For any given $(X, w)\in \ucaln$ of type $\lambda[q]$ with $\lambda=(a_1\geq a_2\geq\ldots \geq a_t)=b_1^{d_1}\ldots b_r^{d_r}$, and $w$ being $\uG$-conjugate $u_j=v_{q+d_j}$ for $q=d_{1}+\cdots+d_{j-1}$ ($j\in\{1,\cdots,r\}$) or $u_{r+1}=0$,
 we define for $k\in \{0,1,\ldots,t\}$, the $k$th enhanced number  is $$\wp_k^{(X,w)}\;(\text{ or } \wp_k^{\lambda[q]}):=\sum_{i=1}^k a_i+\delta^{\lambda}_{k+1,q},$$ where the second summand is defined via
\[
\delta^{\lambda}_{k+1,q}=\begin{cases} a_{k+1} &\quad\text{ if } k+1 > q;\cr
a_{k+1}-1 &\quad\text{ otherwise. }
\end{cases}
\]
If $X=0$,  then the $\uG$-conjugacy class of $(0,w)$ is either the single element $(0,0)$ or $\{0\}\times (V\backslash \{0\})$, the latter of which has a representative element $(0, v_n)$ of type $0[0]$. By definition for any $k=0,1,\ldots,n$, we have $\wp_k^{(0,w)}=k+1$ if $w\ne 0$, and $k$ otherwise.
%then its corresponding standard Jordan form is $(1)_n$. So its %lowerings are either $(1)_n$ itself, or $(0)$. The former corresponds %to the $\uG$-orbit of $0$ when $w=0$. The latter corresponds to the %$\uG$-orbit $(0, V\backslash \{0\})$ consisting of   nonzero vectors.

Now suppose $X (\ne 0)$ has the Jordan standard matrix $J=J_\lambda$ with $\lambda=(a_1\geq a_2\geq\cdots \geq a_t>0)$ reformulated as in \eqref{new lam express}. By Proposition \ref{thm: classi}, $(X,w)$ is preliminarily conjugate to some $(J, v)$, latter of which is finally conjugate to a unique $\uG$-orbit where  $(J, u_j)$ lies, $j=1,\ldots,r, r+1$  (note that there is an appointment $u_{r+1}=0$. So, the occurrence  $w\in \im J$ is equivalent to say that the above $v$ is conjugate to $u_{r+1}=0$ mod $\im J$). We assign  $(J,u_j)$ to correspond to the number $q$ with
\begin{align}\label{eq: q}
q=d_0+ d_1+\cdots+d_{j-1}.
\end{align}
The enhanced number  $\wp_k^{\lambda[q]}$ is only dependent on the enhanced nilpotent orbit $\textsf{Ad}(\uG)(X,w)$, independent of the choice of its representatives.

\subsection{The meaning of enhanced numbers} There is an explanation on the enhanced number $\wp_k^{(X,w)}$ of an enhanced nilpotent orbit $\textsf{Ad}(\uG)(X,w)$. For a set of vectors $\Upsilon:=\{\nu_1,\ldots, \nu_k\}\in V$, and $\tilde w\in w+\im X$, one can consider a $\bk[X]$-submodule of $V$ generated by $\tilde w$ along with $\Upsilon$ which is clearly a $\bk$-subspace spanned by the following vectors
$$\tilde w,X\tilde w,\ldots, X^{n-1}\tilde w; \nu_1, X\nu_1,\ldots,X^{n-1}\nu_1;\ldots; \nu_k, X\nu_k,\ldots,X^{n-1}\nu_k.$$
Denote by $\mathcal{V}_k$ such a vector subspace of $V$ associated with $(X,w)$ and $\Upsilon$. Clearly, $\dim\mathcal{V}_k$ is invariant under $\uG$-action.

By definition, the following inequality holds: $\dim\mathcal{V}_k\leq \wp_k^{(X,w)}$ for any $\tilde w\in w+\im X$, and for any subset $\Upsilon=\{\nu_1,\ldots,\nu_k\}\subset V$.
 In the following, it will be seen that the quality certainly happens for some suitable subset $\Upsilon$.
\begin{lemma}\label{lem: enh numb}
The enhanced number $\wp_k^{(X,w)}$ is the maximal number of $\dim\mathcal{V}_k$ when $\tilde w$ ranges through $w+\im X$ and
$k$-tuple $(\nu_1,\ldots, \nu_k)$ runs over $\overset{k}{\overbrace{V\times V\times\cdots \times V}}$.
\end{lemma}

\begin{proof} If $X=0$, the statement is clear. Suppose $X\ne0$, and $(X,w)$ is of type $\lambda[q]$. As before we write $q=d_0+d_1+\cdots+d_{j-1}$.
Note that $\wp_k^{(X,w)}$ is only dependent on the $\uG$-conjugacy class of $(X,w)$. So we might as well suppose $(X,w)$ to meet \eqref{eq: Jordan st basis 1} and  \eqref{eq: Jordan st basis 2}, specially suppose $w=v_{q+d_{j}}=v_{d_1+\cdots+d_{j}}$.
 By \eqref{eq: Jordan st basis 1} it is clear that $\dim\mathcal{V}_k\leq \sum_{i=1}^{k+1}a_{i}$.  If $q< k+1$, by definition $a_q>a_{k+1}$. In this case, $\wp_k^{(X,w)}=\sum_{i=1}^{k+1}a_i$.
 Hence, $a_{q+d_k}\geq a_{k+1}$. Without loss of generality, we might as well suppose  $q+d_k\leq k+1$. Then take $\nu_i$ running through $\{v_1,\ldots, v_{q+d_j-1}, v_{q+d_j+1},\cdots, v_{k+1}\}$, and take $\tilde w=w$. Then $\dim\mathcal{V}_k= \sum_{i=1}^{k+1}a_i$.

 Suppose $q\geq k+1$. Then $a_{q+d_j}<a_q\leq a_{k+1}$. Correspondingly $\dim k[X]w=a_{q+d_j}<a_{k+1}$. It is clear that the dimension of $\calv_k=\bk[X]\tilde w+\bk[X](\nu_1,\ldots, \cdots, \nu_k)$ is not bigger than $\sum_{i=1}^k a_i+(a_{k+1}-1)$. On the other hand, when we take $\nu_i=v_i$ for $i=1,\ldots,k$, and take $\tilde w$ to be $w+X.v_{k+1}$. Then $\bk[X](\tilde w, v_1,\ldots, v_k)$ attains the maximal dimension $\sum_{i=1}^k a_i+(a_{k+1}-1)$ which is exactly $\wp_k^{(X,w)}$.

The proof is completed.
\end{proof}

\subsection{Explanation by lowerings of a partition $\lambda\in \scrp_n$} \label{sec: lowerings}
The classification of $\uG$-orbits in $\ucaln$ is parameterized by enhanced partitions. The latter can be reformulated by partitions and their lowerings.
\subsubsection{Lowerings of a partition}\label{subsec: lowerings} For a given  partition $\lambda=(a_1,\ldots,a_t)\in\scrp_n$, call $\mu$ a lowering of $\lambda$ if there exists $q\in\bbn$ such that $\mu=\lambda-(1)_q:=(a_1-1,\cdots,a_q-1,a_{q+1},\ldots,a_t)$ is still a partition of $n-q$ where $(1)_q:=(1^q)=(1,\ldots,1)$, a $q$-tuple with all entries equal to $1$ (we make an appointment $(1)_0=0$). For example, for $\lambda=(b_1^{d_1}\ldots b_r^{d_r})$ as in \eqref{new lam express}, $\lambda-(1)_{d_1}=((b_1-1)^{d_1}b_2^{d_2}\ldots b_r^{d_r})$ is a lowering of $\lambda$. It is easily known that when $j$ runs through $\{1,\ldots,r, r+1\}$,
 \begin{align*}
 \lambda-(1)_{d_0+d_1+d_2+\cdots+d_{j-1}}=\begin{cases} \lambda,  &\text{ if }j=1,\cr
 ((b_1-1)^{d_1}\ldots (b_{j-1}-1)^{d_{j-1}} b_{j}^{d_{j}}\ldots b_r^{d_r}), &\text{ otherwise}
 \end{cases}
 \end{align*}
provide all lowerings of $\lambda$ (we make an appointment $d_0=0$). Recall $\lambda[q]$ is  already defined before. Now we endow it with a new meaning by identifying $\lambda[q]$ with $(\lambda, \lambda-(1)_q)$ the pair of a partition and its given lowering $\lambda-(1)_q$.  Denote by $\mathcal{L}_\lambda$ the set of all lowerings of $\lambda$. And let
$\callp_n$ denote the set of all lowerings of partitions of $n$, which means
\begin{align*}
\callp_n&=\bigsqcup_{\lambda\in\scrp_n}\mathcal{L}_\lambda\cr
&=\{ (\lambda, \lambda-(1)_{q})\mid \lambda=(b_1^{d_1}\ldots b_r^{d_r})\in\scrp_n, q\in\{\sum_{i=0}^{j-1}d_i\mid j=1,\ldots,r+1\}\}.
\end{align*}
There is a one-to-one correspondence between $\scrpe$ and $\callp_n$.

\subsubsection{The lowering corresponding to the orbit $\textsf{Ad}(\uG)(X,w)$}\label{rem: by partitions}
%The classification of $\uG$-orbits in $\ucaln$ is compatible with the %$\GL_n$-conjugacy classes in $\ucaln$ in \cite{AH} where the %classification is made by normalized basis determined by a bipartition %$[\mu,\nu]$.
%
%
  For a given $(X, w)\in \ucaln$ with $X=0$, then its corresponding partition is $(1^n)$. So its lowerings are either $(1^n)$ itself, or $(0)$. The former corresponds to the $\uG$-orbit of $0$ when $w=0$. The latter corresponds to the $\uG$-orbit $\{0\}\times (V\backslash \{0\})$ consisting of   nonzero vectors.

Now suppose $X (\ne 0)$ has the Jordan standard matrix $J=J_\lambda$ with $\lambda=(a_1\geq a_2\geq\cdots a_t>0)$ reformulated as in \eqref{new lam express}. By Proposition \ref{thm: classi}, $(X,w)$ is a unique  $\textsf{Ad}(\uG)(J, u_j)$, $j=0,1,\ldots,r$. This $(J,u_j)$ corresponds to the lowering $\lambda-(1)_q$ of $\lambda$ with $q$ as  in \eqref{eq: q}.

\subsection{A partial order of enhanced partitions}
Recall there is a typical partial order in the set $\scrp_n$ of partitions of $n$, which says for any two given $\lambda=(\lambda_1\geq \lambda_2\geq\cdots), \mu=(\mu_1\geq\mu_2\geq\cdots)\in \scrp_n$, $\mu\leq\lambda$ if and only if  $\sum_{i=1}^k\mu_i\leq \sum_{i=1}^k\lambda_i$ for all $k\in \bbz_{>0}$.

\subsubsection{} We turn back to  $\scrpe$.
%first introduce a partial order in $\callp_n$. Keep in mind that we %identify $\lambda[q]$ with $(\lambda,\lambda-(1)_q)$ as in  \S\ref{subsec: %lowerings}.

\begin{defn} \label{defn:partial order}
	%\begin{enumerate}
%		\item
For any given $\lambda[q], \mu[l]\in \scrpe$,
we indicate that $\mu[l]\preceq \lambda[q]$ if and only if the following two items satisfy
\begin{itemize}
\item[(1)] $\mu\leq \lambda$ in $\scrp_n$.
\item[(2)] For any  $k\geq 0$,
		%\begin{align}\label{eq: partial order}
		$\wp_k^{\mu[l]}
		\leq \wp_k^{\lambda[q]}$.		
%\end{align}
\end{itemize}
\end{defn}
\subsubsection{}\label{convering order} Now we are given $\lambda[q]$ with $\lambda=(a_1\geq a_1\geq\cdots)$. Suppose $\mu$ ($\leq \lambda$) is another partition of $n$ with $\mu=(c_1\geq c_2\geq\cdots)$. Let us investigate the possibility of $\mu[l]$ with $\mu[l]$ covered by $\lambda[q]$  for $\lambda\geq \mu$, which means $\mu[l]\preceq\lambda[q]$ while for any $\mu'[l']$ different from $\lambda[q]$,  $\mu[l]\preceq \mu'[l'] \preceq\lambda[q]$ yields $\mu[l]=\mu'[l']$. We denote this strong relation by $\mu[l]\lhd\lambda[q]$. Similarly, we define $\lambda\lhd\mu$. By a straightforward computation, $\mu[l]\lhd\lambda[q]$ implies $\mu\lhd\lambda$.

When $\lambda=\mu$, $\mu[l]\lhd\lambda[q]$ if and only if $l=q+d_j$. Now suppose $\mu\lhd \lambda$. The following facts can be checked by a direct computation.
\begin{itemize}
    \item[(1)] When $q\leq l$, then
     $\mu[l]\lhd\lambda[q]$ if and only if there exists $k\geq j+1$ such $l=d_1+\cdots+d_{k}-1$ and $$\mu=\lambda-\epsilon_{d_1+\cdots+d_k}+\epsilon_{d_1+\cdots+d_k+1}$$
   where $\epsilon_i$ is defined via an alternative expression $\sum_{i=1}^nc_i\epsilon_i$ of a partition $\mu=(c_1\geq c_2\geq\cdots)$.
      \item[(2)] When $q>l$, then $\mu[l]\lhd\lambda[q]$ if and only if $l=q-1$ for  $q=d_1+\cdots+d_{j-1}$,  with $d_{j}=1$ and
      $$\mu=\lambda-\epsilon_{d_1+\cdots+d_{j-1}}+\epsilon_{d_1+\cdots+d_{j}}.$$
      \end{itemize}

%\subsection{Dimensions of enhanced nilpotent orbits}
%\subsection{A remark}

\section{Closures of enhanced nilpotent orbits}
In this section, we will investigate the closures of enhanced nilpotent orbits via establishing partial flags. Some classical strategies (see \cite{AH}, or \cite{BMac2}, \cite{dCP}, \cite{Spa}, {\sl etc.}) are exploited here. Keep the notations and assumptions as in the previous section.

\subsection{Enhanced flag varieties and description of closures of enhanced nilpotent orbits}\label{subsec: on the basis} For a given nilpotent element $(X,w)\in \ucaln$, by the arguments in the previous section we can assume $(X,w)$ lies in an enhanced orbit represented by $\lambda[q]\in \scrpe$.
If $X=0$ and $w=0$, then $(X,w)=(0,0)$ is $\uG$-invariant, and the closure of the corresponding nilpotent orbit is just the single point. When $X=0$ while $w\ne 0$, then {the closure of} the $\uG$-conjugacy class of $(X,w)$ coincides with $V$.

In the following, suppose $X\ne 0$. Suppose $X$ has the Jordan standard matrix $J_\lambda$ with $\lambda=(a_1\ldots a_t)=(b_1^{d_1}\ldots b_r^{d_r})$ as \eqref{new lam express}. Simply denote $J=J_\lambda$. So $(X,w)$ can be represented by $(J, u_j)$ in the enhanced orbit $\co_{\lambda[q]}$ with $j\in\{0,1,\ldots,r\}$  (recall that correspondingly $u_0=0$ and $u_j=v_{d_1+\cdots+d_j}$ while $q=0$ or $q=d_1+\cdots+d_{j-1}$).

\subsubsection{A canonical partial flag associated with $\lambda[q]$} Keep the notations as in \eqref{eq: Jordan st basis 1}, \eqref{eq: Jordan st basis 2} and Lemma \ref{lem: quotient orb}. In particular, keep in mind that $q=d_0+d_1+d_2+\cdots+d_{j-1}$, and $t=d_0+d_1+\cdots+d_r$.
Associated with $\lambda[q]$, let us define a canonical partial flag  over $V$
denoted by $\calf^0$ corresponding to the basis structure in \eqref{eq: Jordan st basis 2}, which is defined via $\calf^0=(F^0_i)_{i=0,1,\cdots,a_1}$ with $F^0_{a_1}=V$, and
\begin{align}\label{eq: canonical flag}
&F^0_{a_1-1}=\im X+\sum_{k=q+1}^{t}\bk v_{k},\cr
&F^0_{i}=\sum_{k=1}^q \bk X^{a_1-i}v_k+\sum_{k=q+1}^t\bk { X^{a_1-k-1}v_k,}\;  i=1,\cdots, a_1-2;\cr
&F^{0}_0:=0.
\end{align}
Clearly, $\dim F^0_{a_1-1}=\dim\im(X)+(t-q)=n-q$, and
$$\dim F^0_{i}=\dim \im(X^{a_1-i})=n-q-\sum_{k=i+1}^{a_1-1} ((\lambda-(1)_q)^\sft)_{a_1-k}$$ for $i=1,\cdots, a_1-2$,
 where $\lambda-(1)_q$ still denotes the lowering of $\lambda$, and $(-)^\sft$ denotes the transpose of a partition $(-)$. We simply set $m_i:=\dim F^0_{i}$,  then denote $n_i:=m_i-m_{i-1}$. Clearly $n_i=((\lambda-(1)_q)^\sft)_{a_1-i}$ for $i=1,\cdots,a_1-1$,
  and $n_{a_1}=q$. Here and further, keep in mind the meaning of notation $\mu_i$ which stands for  the $i$th component of a partition $\mu$. This is to say, $\mu=(\mu_1,\mu_2,\ldots)$.

\begin{lemma}\label{lem: trans of parts} The following statements hold.
\begin{itemize}
\item[(1)] The subspace $F^0_{a_1-1}$ coincides with $M$ in Lemma \ref{lem: about M}.
\item[(2)] The set $\{n_i\mid i=1,\ldots,a_1\}$ coincides with $\{\lambda^\sft_i\mid i=1,\ldots,a_1\}$.
    \end{itemize}
\end{lemma}

\begin{proof} (1) By the proof of Lemma \ref{lem: about M}, $F^0_{a_1-1}\subset M$. The statement follows from the comparison of the dimensions.

(2) It is readily known from the construction of the filtration $\{F^0_{i}\}$.
\end{proof}

\subsubsection{Enhanced flag varieties}
Next we can define a partial flag variety $\mathscr{F}_{\lambda[q]}$
  on $V$ consisting of subspace filtrations maintaining the same dimensions as in the flag
  \eqref{eq: canonical flag}. This is to say:
  \begin{align*}
\mathscr{F}_{\lambda[q]}=\{(M_i)
:=(0=:M_{0}\subset M_{1}\subset\ldots\subset M_{a_1-1}\subset M_{a_1}:= V)\mid
 \dim M_{i}=m_i\}.
\end{align*}
It is clearly {both $G$-stable and $\uG$-stable.}

Then we introduce a so-called associated flag variety
\begin{align}
\scraf_{\lambda[q]}=\{(Y, (M_i))\in \caln\times \scrf_{\lambda[q]}
\mid (M_i)\in\scrf_{\lambda[q]}, \;Y(M_i)\subset M_{i-1} \}
\end{align}
and  so-called enhanced flag variety
\begin{align}
\scref_{\lambda[q]}=\{((Y,u), M_i)\in \ucaln\times \scrf_{\lambda[q]}\mid u\in M_{a_1-1}, \;{(Y, (M_i))}\in \scraf_{\lambda[q]} \}.
\end{align}
%Note that $G_X$ acts naturally on $\widetilde V$. So $\scrf_{\lambda[q]}$ is %$G_X$-equivariant.

{

\begin{remark} When $\lambda=(n)$ which corresponds to the regular nilpotent orbit, in the meanwhile $q=0$, then $\scrf_{\lambda[q]}$ is actually a full flag variety over $V$, $\scrf_{\lambda[q]}$ is actually isomorphic to the cotangent bundle of $\scraf_{\lambda[q]}$. In this case, $\scrf_{\lambda[q]}$, $\scraf_{\lambda[q]}$ and $\scref_{\lambda[q]}$ will be simply write $\scrf$, $\scraf$ and $\scref$, respectively.
 \end{remark}

}

We first have the following observation.
Turn back to the standard partial flag $\mathcal{F}^0$. Clearly, the $G$-saturation $G.\calf^0$ is exactly  the whole $\scrf_{\lambda[q]}$. Let $G_{\calf^0}$ denote the stabilizers, i.e. $G_{\calf^0}=\{g\in G\mid g(\calf^0)=\calf^0\}$. Then $G_{\calf^0}$ is certainly a parabolic subgroup of $G$.
Denote by $U^0$ the unipotent radical of $G_{\calf^0}$.
%Then $\Lie(U^0)=\sum_{k>0}\gl_n^{(k)}$ as defined in \S\ref{subsec: enh conj notations}.  % Still set $\im(X)=X(V)$ the %subspace of $V$.
By construction we have an isomorphism
$$ G_{\calf^0}\slash U^0\cong \prod_{i=1}^{a_1}\GL_{n_i}.$$

Furthermore, we have the following fact.

\begin{lemma}\label{lem: parabolic lem}
The stabilizer $G_{\calf^0}$ is isomorphic to P as defined in \S\ref{subsec: enh conj notations}.
\end{lemma}
\begin{proof} By definition, $G_{\calf^0}$ is a parabolic subgroup of $G$ with the same quotient by the unipotent radical as $P$. The statement follows.
\end{proof}

So, we have
$$\dim\scrf_{\lambda[q]}=\dim G-\dim G_{\calf^0}{=}{1\over 2}(\dim G-\dim(G_{\calf^0}\slash U^0)),$$
consequently
$$\dim\scrf_{\lambda[q]}={1\over 2}(n^2-\sum_{i=1}^{a_1} n_i^2)\overset{Lemma \ref{lem: trans of parts}(2)}{=}{1\over 2}(n^2-\sum_{i=1}^{a_1}(\lambda^\sft)_i^2).$$

Consider the projection $\scraf_{\lambda[q]}\rightarrow \scrf_{\lambda[q]}$, which is a vector bundle.  Its fiber at $\calf^0$ is just $\Lie(U^0)$. This vector bundle is actually the fiber bundle $G \times^{G_{\calf^0}} \Lie(U^0)$ (see \cite[\S5.14]{Jan1} for the notation). Correspondingly we have
\begin{align}\label{eq: flag dim a}
\dim \scraf_{\lambda[q]}=2\dim U^0=n^2-\sum_{i=1}^{a_1}(\lambda^\sft)_i^2=\dim \co_\lambda.
\end{align}
Similarly, the projection $\scref_{\lambda[q]}\rightarrow \scraf_{\lambda[q]}$ is also a vector bundle. Furthermore, its fiber at $(X, \calf^0)$ is just the vector subspace of {$\dim M_{a_1-1}$} equal to $n-q$. Hence
\begin{align}\label{eq: flag dim e}
\dim \scref_{\lambda[q]}=n^2-\sum_{i=1}^{a_1} n_i^2 {+n-q}.
\end{align}

\begin{lemma}\label{lem: two proj prop} The following statements hold
\begin{itemize}
\item[(1)] Both $\scraf_{\lambda[q]}$ and $\scref_{\lambda[q]}$ are $\uG$-equivariant.
\item[(2)]
Both $\scraf_{\lambda[q]}$ and $\scref_{\lambda[q]}$ are smooth irreducible varieties.
\end{itemize}
\end{lemma}

 \begin{proof}
 The first statement is clear. As to the second one, it follows from the arguments before the lemma.
\end{proof}

Consider the following two projections  on the first coordinators respectively:
$$\pia: \scraf_{\lambda[q]}\rightarrow  \caln$$
and
\begin{align}\label{eq: enh proj}
\pie:\scref_{\lambda[q]}\rightarrow \ucaln.
\end{align}

\begin{lemma}\label{lem: flag closure} Keep the notation as before. In particular, let $(X,w)\in \ucaln$ be of type $\lambda[q]$.
 The following statements hold.
\begin{itemize}
\item[(1)] Both $\pia$ and $\pie$ are proper.

\item[(2)] The image of $\pia$ is exactly $\overline{\co_{\lambda}}$, the closure of the $G$-orbit $\co_{\lambda}$ of $X$ in $\caln$.

\item[(3)] The image of $\pie$ is exactly $\overline{\co_{\lambda[q]}}$, whose dimension is equal to $\dim \overline{\co_\lambda}+(n-q)$.

    \end{itemize}
\end{lemma}

\begin{proof} (1) Note that $\scrf_{\lambda[q]}$ is a projective variety. The partial flag varieties $\scraf_{\lambda[q]}$ and $\scref_{\lambda[q]}$ are close subvarieties of { $\caln\times \scrf_{\lambda[q]} $ and $\ucaln\times \scrf_{\lambda[q]}$} respectively. The first statement follows.

(2) %Note that all nilpotent orbits for $\GL_n$ are Richardson orbits. %By the proof of \cite[Theorem 5.2.3]{Car} along with  Lemma \ref{lem: %parabolic lem}, we have $\dim \overline{\co_{\lambda}}=2\dim U^0$. So
From \eqref{eq: flag dim a} we have
\begin{align}\label{eq: flag dim eq}
\dim\scraf_{\lambda[q]}=\dim\overline{\co_{\lambda}}.
\end{align}

Note that $\pia$ is $G$-equivariant. By Lemma \ref{lem: two proj prop}  $\im\pia$ is an irreducible closed $G$-variety in $\caln$ which contains finite $G$-orbits. So $\im\pia$ must be the closure of some nilpotent orbit in $\caln$ under adjoint $G$-action. In the meanwhile, $\co_{\lambda}$ is clearly contained in $\im\pia$. Hence $\dim\im\pia\geq \dim\overline{\co_\lambda}$. So the statement follows from  \eqref{eq: flag dim eq}.

(3) We  take the same arguments as in (2). By \eqref{eq: flag dim e} and the second statement (2), we have $\dim \scref_{\lambda[q]}=\dim\overline{\co_\lambda}+(n-q)$. Taking Proposition \ref{prop: enh orbit dim} into an account, we further have $\dim\scref_{\lambda[q]}=\dim\overline{\co_{\lambda[q]}}$. Now $\pie$ is $\uG$-equivariant, and $\im\pie$ is a $\uG$-equivariant irreducible closed subset of $\ucaln$. By Theorem \ref{thm: classi}, $\ucaln$ contains finite $\uG$-orbits. Hence $\im\pie$ must be a closure of some $\uG$-orbit which contain $\overline{\co_{\lambda[q]}}$. Hence $\im\pie$ coincides with $\overline{\co_{\lambda[q]}}$.
\end{proof}

\begin{corollary}\label{cor: closure M}
Keep the notations and assumptions as before. In particular, $M=\im X+\ggg_X.w$ as in Lemma \ref{lem: about M}. Then the following statements hold
\begin{itemize}
\item[(1)] The closure $\overline{\co_{\lambda[q]}}$ coincides with the $\uG$-saturation $\textsf{Ad}(\uG)(\Lie(U^0)\times M)$.
\item[(2)] The projection $\pie$ is a resolution of singularities of $\overline{\co_{\lambda[q]}}$.
    \end{itemize}
\end{corollary}
\begin{proof} (1) For any $(Y,u)\in \Lie (U^0)\times M$, $((Y,u), \calf^0)$ is a partial flag in $\scref_{\lambda[q]}$. By Lemma \ref{lem: flag closure}(3), $(Y,u)\in \overline{\co_{\lambda[q]}}$. Hence $\textsf{Ad}(\uG)(\Lie (U^0)\times M)$ is contained by $\overline{\co_{\lambda[q]}}$.

Conversely, for any $(Y,u)\in \overline{\co_{\lambda[q]}}$, by Lemma \ref{lem: flag closure}(3) again there exits $\calf=(M_i)\in \scrf_{\lambda[q]}$ such that $((Y,u), \calf)\in \scref_{\lambda[q]}$.
Thus, there exists $g\in G$ such that $g. \calf=\calf^0$, correspondingly
$\textsf{Ad}(g)u\in F^0_{r-1}=M$ and  $\Ad(g)(Y)\in \Lie(U^0)$. Thus $\textsf{Ad}(g)(Y,u)\in \Lie(U^0)\times M$.

(2) By Lemma \ref{lem: two proj prop}(2) and Lemma \ref{lem: flag closure}(2), it suffices to show that the restriction of $\pie$ to $\pie^{-1}(\co_{\lambda_[q]})$ is an isomorphism onto $\co_{\lambda[q]}$.

 We first claim that the restriction of $\pia$ to $\pia^{-1}(\co_{\lambda})$ is an isomorphism onto ${\co_{\lambda}}$.
Recall that for a given $X\in \co_\lambda$, there is a cocharacter $\tau$ of $G$ such that $\tau$ gives rise to a gradation of $\ggg$ compatible with the gradation of $V$ arising from a Jordan basis associated with $X$  (see \S\ref{subsec: enh conj notations}, or \cite[\S5.3]{jan2}). Furthermore, one associates to $\tau$ a parabolic subgroup $P$, and its Lie algebra $\ppp$ whose unipotent radical is denoted by $\uuu$. By Richardson's orbit property for $\GL_n$ (see \cite[\S5]{Car} or \cite[\S4.9]{jan2}), $\co_\lambda$ is a unique nilpotent orbit satisfying that its intersection with $\uuu$ is dense in $\uuu$. Furthermore,  $\co_{\lambda}$ is exactly the $G$-saturation of this intersection, and $P$ is $G$-conjugate to $G_{\calf^0}$. Correspondingly, there is a unique partial flag  $\calf\in \scraf$  such that $(X, \calf)\in \pia^{-1}(X)$ is exactly the unique inverse image of $X$. This
$\calf$ is conjugate to $\calf^0$. Furthermore, the map $\co_{\lambda}\rightarrow \scraf$ with $X\mapsto \calf$ is a morphism. This is because  $\co_\lambda\cong G\slash G_{\calf^0}$, and the map $X=\textsf{Ad}(g)J_\lambda\mapsto \calf=g.\calf^0$ is actually a morphism of varieties from $G\slash G_{\calf^0}$ to $G\slash P$.

Next, we turn back to the part of $\pie$. From the above claim on $\pia$, it follows that for any $(X, w)\in \co_{\lambda[q]}$ there exits a unique triple $((X,w), \calf)\in \pie^{-1}(X,w)$, and furthermore the map $(X,w)\mapsto \calf$ is a morphism of varieties from $\co_{\lambda[q]}$ to $\scref$.

 The proof is completed.
\end{proof}

\begin{remark} \label{rem: simple notations}
(1) By the above proof, we actually have that the projection $\pia$ is a resolution of singularities of $\overline{\co_{\lambda}}$.

(2)  When $\lambda=(n)$ corresponds to the regular nilpotent orbit, and $q=0$, $\pie$ is actually a resolution of singularities of the whole $\ucaln$ associated with the enhanced reductive group $\uG$. In this case, we directly write it as $\pies$, and call it the enhanced Springer resolution. Similarly, we directly write $\pia$ as $\pias$. This $\pias$ is exactly the ordinary Springer resolution.

(3) In view of Proposition \ref{prop: comparison},  $\overline{\co_{\lambda[q]}}$ can be described by Achar-Henderson's way in \cite{AH}, where the partial flags are constructed via bi-partitions and  related divisions  of  Jordan basis.
\end{remark}

\subsection{An alternative description of the closures of enhanced nilpotent orbits}
Consider a  subvariety of $\ucaln$ consisting of $(Y, u)$ subject to the following condition for any $k\in\{1,2,\ldots, n\}$
\begin{align}\label{eq: closure cond}
\dim \bk[Y](\tilde u,\nu_1,\nu_2,\ldots,\nu_k)\leq \wp_k^{\lambda[q]}\;\text{ for any }\tilde u\in u+\im Y, \text{ and } \nu_1,\nu_2,\ldots,\nu_k\in V.
\end{align}
Then for each $k$, the above condition determines a closed subvariety of $\ucaln$. Hence we have a closed subvariety $\ucaln_{\lambda[q]}$ as the intersection of such closed subvarieties when  $k$ ranges over $\{1,\ldots,n\}$.

\begin{lemma}\label{lem: invar closed set} The closed subvariety  $\ucaln_{\lambda[q]}$ is $\uG$-invariant. Consequently,  $\ucaln_{\lambda[q]}\supset \overline{\co_{\lambda[q]}}$.
\end{lemma}
\begin{proof}  Note that $V$ is invariant under $\uG$-action. Hence  $\ucaln_{\lambda[q]}$ is $\uG$-invariant. The second part follows from Lemma \ref{lem: enh numb}.
\end{proof}

On the other hand, we have the following fact.
\begin{lemma}\label{cor: pre equiv} Suppose $(Y,u)\in\ucaln$ is of type $\mu[l]$. If $Y$ is already in $\Lie(U^0)$ and $\wp_k^{(Y,u)}\leq \wp_k^{(X,w)}$ for any $k\geq 0$, then $(Y,u)\in \overline{\co_{\lambda[q]}}$.
\end{lemma}

\begin{proof} For a given $(Y,u)\in\ucaln$ of type $\mu[l]$ with $Y\in\Lie(U^0)$ and with $\wp_k^{(Y,u)}\leq \wp_k^{(X,w)}$ ($k\geq 0$), in order to show $(Y,u)\in \overline{\co_{\lambda[q]}}$,
by Corollary \ref{cor: closure M} it suffices to show that $u\in M$.

    We can write $u=\sum_{i=1}^q \sum_{k=0}^{a_i-1} c_{ik}X^kv_i+\sum_{i=q+1}^{t}
    \sum_{k=0}^{a_k-1}h_{ik}X^kv_i$ where $c_{ik}, h_{ik}\in \bk$.
The second sum belongs to $\ggg_X.w$ (Lemma \ref{lem: quotient orb}). So $u$ can be rewritten as below:
$$ u=\sum_{i=1}^q \sum_{k=0}^{a_i-1} c_{ik}X^kv_i + u^*$$
for $u^*\in \ggg_X.w$. We claim $\sum_{i=1}^q \sum_{k=0}^{a_i-1} c_{ik}X^kv_i\in \im X$.
 Actually, consider the space
 $$\bk[Y](\tilde u, \nu_1,\nu_2,\ldots,\nu_k)$$ for $\tilde u\in u+\im Y$, $\nu_1,\ldots,\nu_k\in V$. Recall the assumption $Y\in \Lie(U^0)$ and $\wp_k^{(Y,u)}\leq \wp_k^{(X,w)}$ ($k\geq 0$). By the analysis in \S\ref{convering order}, all $c_{ik}$ must be $0$ for $k=0$ when $i=1,\ldots, q$. The proof is completed.
\end{proof}

Summing up, we have the following result.
\begin{prop}\label{prop: closure}
 $\overline{\co_{\lambda[q]}}=\ucaln_{\lambda[q]}\cap (\overline{\co_\lambda}\times V)$.
\end{prop}

\begin{proof} The  part ``$\subseteq$" follows from  Corollary \ref{CorBasicClosure} and Lemma \ref{lem: invar closed set}.

Next we prove the inverse inclusion. For any given $(Y,u)\in \ucaln_{\lambda[q]}\cap (\overline{\co_\lambda}\times V)$, we can suppose $(Y,u)$ is of type $\mu[l]$. By the assumption, $\co_{\mu}\subset \overline{\co_{\lambda}}$. Note that by Richardson's orbit property for $\GL_n$, $\overline{\co_{\lambda}}$ coincides with the closure of the $G$-saturation  $\textsf{Ad}(G)(\co_{\lambda}\cap \Lie(U^0))$, i.e. $\overline{\co_{\lambda}}=\textsf{Ad}(G)\Lie(U^0)$ (see \cite{Car}). Hence, up to $G$-conjugacy, we might as well suppose $Y\in\Lie(U^0)$. The remaining follows from Lemma \ref{cor: pre equiv}.
\end{proof}

\subsection{Topology relations of the closures of enhanced nilpotent orbits}
Now we can conclude the closure relation of enhanced nilpotent orbits as  in the ordinary case (see \cite[\S6.3]{CM}, \cite{jan2}). Suppose $(Y, u)\in \caln$ of type $\mu[l]$, then it is deduced  that $\mu[l]\preceq\lambda[q]$, which is compatible with  \cite[Theorem 6.3]{AH}. This means

\begin{prop}\label{prop: closure rel} Suppose $(Y,u)\in \ucaln$  is of type $\mu[l]$. Then $(Y,u)$ lies in $\overline{\co_{\lambda[q]}}$ if and only if $\mu[l]\preceq\lambda[q]$.
\end{prop}
\begin{proof} It is a corollary to Proposition \ref{prop: closure}.
\end{proof}

\subsection{Comparison with $\GL_n$-orbits in $\ucaln$}\label{sec: comparison}

The enhanced nilpotent cones are both adjoint $\uG$- and adjoint $G$-stable affine varieties. There are finite orbits in $\ucaln$ in both cases under adjoint $\uG$-action and adjoint $G$-action.
In general, for a given $(X,w)\in \ucaln$, two orbits $\textsf{Ad}(\uG) (X,w)$ and $\textsf{Ad}(G)(X,w)$ are not necessarily the same. For example, $\dim \textsf{Ad}(G)(X,0)=\dim \textsf{Ad}(G)X$, however $\dim \textsf{Ad}(\uG)(X,0)=\dim \textsf{Ad}(G)X +\dim\im X$
(Corollary \ref{CorBasicClosure}(3)). So $\textsf{Ad}(\uG)(X,0)\supsetneqq
\textsf{Ad}(G)(X,0)$ whenever $X\ne 0$.

It is a natural question what the relation between two types of nilpotent orbits is. For any given $(X,w)\in \ucaln$, set $\co_{(X,w)}=\textsf{Ad}(\uG)(X,w)$ and $\sfO_{(X,w)}=\textsf{Ad}(G)(X,w)$.
By the constructions of enhanced nilpotent orbits under adjoint $\uG$-action and adjoint $G$-action respectively, we have
$$\co_{(X,w)}=\sfO_{(X,w)}\cup\bigcup_{\overset{g\in G}{u\in \im(gXg^{-1})}}\sfO_{(gXg^{-1}, u+gw)}$$

Furthermore, we have the following precise description on the relation  between them.

\begin{prop}\label{prop: comparison} %With respect to the closures of enhanced nilpotent orbits %in Zariski topology, the following statements hold.
\begin{itemize}
\item[(1)] For any given $(X,w)\in \ucaln$ of type $\lambda[q]$ as in Proposition \ref{thm: classi}, set $\hat w=\sum_{i=1}^{n-q}w_i$ where $w_i$ ($i=1,\ldots, n-q$) is a canonical basis of $\im X+\ggg_X.w$ such that $\{w_i\mid i=1,\ldots, n-q\}=\{X^{j_i}v_i\mid i=1,\cdots,t; j_i=1,\ldots, a_i-1 \text{ for }i=1,\ldots,q; \text{ and }j_i=0,1,\ldots, a_i-1 \text{ for }i=q+1,\ldots,t\}$ (the notations as in \eqref{eq: Jordan st basis 2}),  then
    \begin{align}\label{eq: compatible tcl}
    \overline{\co_{(X,w)}}=\overline{\sfO_{(X,\hat w)}}.
    \end{align}
    Equivalent to say,
    $$\overline{\co_{\lambda[q]}}=\overline{\sfO_{(\lambda-(1)_q, (1)_q)}}$$
    where $(\lambda-(1)_q, (1)_q)$ is a bipartition of $n$, and $\sfO_{(\lambda-(1)_q, (1)q)}$ is termed  the $G$-orbit parameterized by  ${(\lambda-(1)_q, (1)_q)}$ in the classification of $G$-orbits  of $\ucaln$ in the sense of \cite[Proposition 2.8]{AH}.

    \item[(2)] With respect to the closures of nilpotent orbits in $\ucaln$, the classification of $\uG$-orbits in Theorem
        \ref{thm: classi} is compatible with the classification of $G$-orbits by bipartitions in \cite{AH}.
\end{itemize}
\end{prop}

\begin{proof} (2) is a consequence of (1).

For (1), we will verify \eqref{eq: compatible tcl}. By definition along with the same arguments as in the proof of Lemma \ref{lem: quotient orb},  $(X,\im X+\ggg_X.w)$ is contained in $\textsf{Ad}(G_X\ltimes V)(X,w)$. So the right hand side of \eqref{eq: compatible tcl} is contained by the left hand side. It is readily known that $\sfO_{(X, \hat w)}$ is exactly $\sfO_{(\lambda-(1)_q,(1)_q)}$ by the classification in \cite[Proposition 2.8]{AH}.
Furthermore, it is deduced by the same reason {\sl{ibid}} that dimension of $\sfO_{(\lambda-(1)_q, (1)q)}$ is equal to $\dim\co_{\lambda}+(n-q)$, which equals the dimension of $\co_{\lambda[q]}$ (see Proposition \ref{prop: enh orbit dim}).
Hence, both  sides are irreducible closed $G$-stable subset of $\ucaln$ with the same dimension. So \eqref{eq: compatible tcl} follows.
\end{proof}

%\newpage
%\section{Normality of closures of enhanced nilpotent orbits}

\section{Intersection cohomology associated with enhanced nilpotent orbits and enhanced Springer fibers}

From now on, we will fix a prime number $\ell$ which is different from $p$ if $\text{ch}(\bk)=p>0$. Let $\overline{\bbq_\ell}$ be the algebraic closure of $\ell$-adic number field.
Recall that in the ordinary case, the study of Springer fibers (such as $\cb_X:=\pias^{-1}(X)$ in type $A$) plays a very important role in combinatorics and geometric {representation} theory. Especially, there is a so-called Springer representation of the Weyl group $W$ on $H^i(\cb_X,\overline{\bbq_\ell})$. This kind of representations were first constructed by Springer (see \cite{Spr}, \cite{BMac}, \cite{Lu} {\sl etc.} or the survey articles  \cite{Sho}, \cite{jan2}).  Enhanced Springer fibers naturally arise here.

In this section, we focus on the intersection cohomology decomposition on $\ucaln$ along the closures of enhanced nilpotent orbits, along with enhanced Springer fibers.
Keep the notations and assumptions as in the previous section. In particular, let $(X,w)\in \ucaln$ be a given enhanced nilpotent element of type $\lambda[q]$. We denote  $\cb_{(X,w)}=\pies^{-1}(X,w)$, which is called the enhanced Springer fiber of $(X,w)\in \ucaln$.

%\subsection{} We first recall some  ordinary nilpotent orbit theory. In %particular,  Spaltenstein's theorem \cite{Spa} on irreducible components %of the fixed point set of $X\in\caln$ on the flag variety over $V$ for %$\GL(V)$.

\subsection{Enhanced Springer fibers}  Keep the notations as before. In particular, keep the conventions as in Remark \ref{rem: simple notations}(2).
Let $(X,w)\in\ucaln$ be arbitrarily given and of type $\lambda[q]\in\scrpe$.
We first look at the enhanced Springer fiber $\cb_{(X,w)}=\pies^{-1}(X,w)$. By definition, we have $\pies^{-1}(X,w)=\pias^{-1}(X)$.  Hence we can identify $\cb_{X,w}$ with $\cb_X$ where $\cb_X=\pias^{-1}(X)$ is the ordinary Springer fiber. So we have the facts

\begin{lemma} \label{thm: 4.3} The following statements hold.
\begin{itemize}
\item[(1)]  An enhanced Springer fibers $\cb_{(X,w)}$ with $(X,w)\in \ucaln$ coincides with the corresponding Springer fiber $\cb_X$. Hence all $\cb_{(X,w)}$ are connected. Its irreducible components have the same dimension.
\item[(2)] $\dim \cb_{(X,w)}={1\over 2}(n^2-n-\dim\co_X)$.
%\item[(3)] Suppose $(X,w)\in\ucaln$ is of type $\lambda[q]$. Set %$l_{(X,w)}=n(n-1)-\dim\co_X+q$. Then $$\dim %H^{l_{(X,w)}}(\cb_{(X,w)},\overline{\bbq_\ell})=0$$ except $(X,w)$ is %of type $\lambda[0]$ with $\lambda=(n)$ while in this exceptional %case, $l_{(X,w)}=0$ and
 %   $$\dim H^0(\cb_{(X,w)}, \overline{\bbq_\ell})=1.$$
\item[(3)] $H^i(\cb_{(X,w)},\overline{\bbq_\ell})=0$ for all $(X,w)\in \ucaln$ if $i$ is odd.
\end{itemize}
\end{lemma}
\begin{proof} The first part of Statement (1) is easily observed from the definitions of $\pies$ and $\pias$. The other parts of (1) along with the second and third statements  follow from the classical results on Springer fiber (see \cite{Spa}, or \cite[\S10, \S11.1 and \S11.5]{jan2}.
 \end{proof}

\begin{remark} There is a conjecture by Achar-Henderson-Jones in \cite[1.2]{AHJ} which  thanks to  \eqref{eq: compatible tcl}, implies  that  all closures of  enhanced nilpotent orbits are normal varieties. With a connection to this conjecture, we have the following proposal.

\begin{conj} Suppose $(X,w)\in \ucaln$ is of type $\lambda[q]$. Then for any  $(Y,u)\in \overline{\co_{(X,w)}}$, the fiber $\pie^{-1}{(Y,u)}$  is connected. All of the irreducible components of this fiber have the same dimension, analogous of Spaltenstein's theorem on the Springer fibers (\cite{Spa}).
\end{conj}
\end{remark}
%\subsubsection{Affine paving of $\cb_{(X,w)}$}

%\begin{prop}

%\end{prop}

%\begin{corollary} %The following statements hold.
%\begin{itemize}
%\item[(1)] $H^{2i}(\cb_{(X,w)}, \overline{\bbq_\ell})=$, and
%\item[(2)]
%$\dim H^{2i}(\cb_{(X,w)}, \overline{\bbq_\ell})=$ and %$H^{2i+1}(\cb_{(X,w)}, \overline{\bbq_\ell})=0$.

%\end{itemize}
%\end{corollary}

\subsection{Enhanced Springer resolution $\pies$ and associated intersection cohomology decomposition} Let us first turn back to the resolutions $\pie$ of singularities of enhanced nilpotent orbit closures (Corollary \ref{cor: closure M}).  Especially, $\pies$ is a resolution of singularities of the whole enhanced nilpotent cone $\ucaln$ (see Remark \ref{rem: simple notations} for the notation).

%For any given enhance nilpotent orbit $\co_{\lambda[q]}$, by Corollary %\ref{cor: closure M} we have that $\pie:\scref_{\lambda[q]}\rightarrow %\overline{\co_{\lambda[q]}}$ is a resolution of singularities.

The following enhanced version of Borho-MacPherson decomposition theorem is partially a counterpart of \cite[Theorem 4.5]{AH} (in a special case) and  of \cite[(11)]{FGT}.
\begin{theorem}\label{thm: 4.1}
 Keep the notations and assumptions as above.  The following statements hold.
\begin{itemize}
\item[(1)] The enhanced Springer resolution $\pies:\scref\rightarrow\ucaln$ is semismall.
\item[(2)] There is a $\uG$-equivariant isomorphism of semisimple perverse sheaves:
        \begin{align}\label{eq: ic decomp}
        \pies_*\overline{\bbq_\ell}[-n^2] \cong \bigoplus_{\lambda\in\scrp_n}V_{(\overline{\co_{\lambda[0]}},
    \overline{\bbq_\ell})}\otimes\text{IC}
    (\overline{\co_{\lambda[0]}},\overline{\bbq_\ell})[-
    (\dim\co_\lambda+n)],
        \end{align}
        where the $V_{(\overline{\co_{\lambda[0]}},
    \overline{\bbq_\ell})}$ are finite-dimensional vector spaces over $\overline{\bbq_\ell}$, the dimension of $V_{(\overline{\co_{\lambda[0]}},
    \overline{\bbq_\ell})}$ is equal to the multiplicity of $\text{IC}
    (\overline{\co_{\lambda[0]}},\overline{\bbq_\ell})[-
    (\dim\co_\lambda+n)]$ as a direct summand in the decomposition
of $\calk$  into a direct sum of simple perverse sheaves, $\calk$ stands for the left hand side in the isomorphism \eqref{eq: ic decomp}.
\item[(3)] Let $\frak{S}_n$ be the Weyl group of $\GL_n$.
 The above decomposition is compatible with $\frak{S}_n$-action. Precisely speaking, one has an action of $\frak{S}_n$ on $\calk$ which induces an algebra isomorphism from $\overline{\bbq_\ell}[\frak{S}_n]$ onto $\End_{\text{Perv}(\ucaln)}(\calk)$, where $\text{Perv}(\ucaln)$ denotes the category of all perverse sheaves on $\ucaln$ which is an abelian category.
\end{itemize}

\end{theorem}
\begin{proof} (1) Keep in mind that $\ucaln=\caln\times V$ and then $\dim\ucaln=\dim\caln+n=n^2$. The first statement is equivalent to claim that for any $i\in \bbn$,
\begin{align}\label{eq: semis}
\dim\{(X,w)\in \ucaln\mid \dim \pies^{-1}(X,w)\geq {i\over 2}\}\leq {n^2}-i.
\end{align}
By definition, it is readily shown that $\pies^{-1}(X,w)$ coincides with $\pias^{-1}(X)$. Therefore,
 $$\{(X,w)\in \ucaln\mid \dim \pies^{-1}(X,w)\geq {i\over 2}\}
 =\{{X\in \caln}\mid \dim \pias^{-1}X\geq {i\over 2}\}\times V.
 $$
Thanks to the classical result, the Springer resolution $\pias$ is semi-small. Hence
$\dim\{X\in \caln\mid \dim \pias^{-1}(X)\geq {i\over 2}\}\leq {n^2-n-i}$. So \eqref{eq: semis} follows.

(2) %Note that $\ucaln=\caln\times V$ and then %$\dim\ucaln=\dim\caln+n=n^2$.
 Recall the centralizer $\uG_{(X,w)}$ of any given $(X,w)\in\ucaln$ in $\uG$ is connected (Lemma \ref{lem: basic lem}). Thus, on the basis of (1) along with Theorem \ref{thm: classi} and \eqref{eq: compatible tcl} we have, as a $\GL_n$-equivariant perverse sheaf, the following Borho-MacPherson decomposition
\begin{align}\label{eq: Lu-Sp dec}
 \pies_*\overline{\bbq_\ell}[-n^2] \cong \bigoplus_{\lambda[q]\in\scrpe}V_{(\overline{\co_{\lambda[q]}},
    \overline{\bbq_\ell})}\otimes\text{IC}
    (\overline{\co_{\lambda[q]}},\overline{\bbq_\ell})[-
    (\dim\co_{\lambda[q]})]
    \end{align}
(see \cite{BMac}). In particular, $\text{IC}
    (\overline{\co_{\lambda[q]}},\overline{\bbq_\ell})[-
    (\dim\co_{\lambda[q]})]$ are the simple perverse sheaves on $\co_{\lambda[q]}$ while  the multiplicity space $V_{(\overline{\co_{\lambda[q]}},
    \overline{\bbq_\ell})}$ could be zero. Furthermore,  this decomposition is $\uG$-equivariant. Actually, $\uG$ is a connected algebraic group, and by construction $\calk$ is a $\uG$-equivariant perverse sheaf on $\ucaln$ which is actually semi-simple perverse because $\ucaln$ has finite $\uG$-orbits. So any subquotient of $\calk$ is $\uG$-equivariant. Hence the decomposition \eqref{eq: Lu-Sp dec} is really $\uG$-equivariant (see \cite[Lemma 12.19]{jan2}).

    We claim that
    \begin{align}\label{eq: claim ic decomp}
    V_{(\overline{\co_{\lambda[q]}},\overline{\bbq_\ell})}=0
          \; \text{ in } \eqref{eq: Lu-Sp dec} \text{ unless }q=0,
          \end{align}
consequently yielding the desired decomposition \eqref{eq: ic decomp}.
     In order to show this, we arbitrarily take $(X,w)\in \ucaln$ and assume it is of type $\lambda[q]$. By definition we have the isomorphism of cohomology $$H^{i}(\cb_{(X,w)},\overline{\bbq_\ell})\cong \calh^{i-n^2}_{(X,w)}(\calk)$$
where $\calh^{j}_{(X,w)}(-)$ is the stalk at $(X,w)$ of the $j$th cohomology of a perverse sheave $(-)$. Thus, from \eqref{eq: Lu-Sp dec} it follows that
$$H^{i}(\cb_{(X,w)},\overline{\bbq_\ell})\cong
\bigoplus_{\mu[l]\in\scrpe}V_{(\overline{\co_{\mu[l]}},
    \overline{\bbq_\ell})}\otimes
    \calh^{i-n^2+\dim\co_{\mu[l]}}_{(X,w)}\text{IC}
    (\overline{\co_{\mu[l]}},\overline{\bbq_\ell}).$$
    Taking $d=n^2-\dim \co_{\lambda[q]}$ and setting  $d_\mu:=\dim\co_{\mu[l]}-\dim\co_{\lambda[q]}$,
     we have
    $$ H^{d}(\cb_{(X,w)},\overline{\bbq_\ell})\cong
\bigoplus_{\mu[l]\in\scrpe}V_{(\overline{\co_{\mu[l]}},
    \overline{\bbq_\ell})}\otimes
    \calh^{d_\mu}_{(X,w)}\text{IC}
    (\overline{\co_{\mu[l]}},\overline{\bbq_\ell}).$$
Suppose $\calh^{d_\mu}_{(X,w)}\text{IC}
(\overline{\co_{\mu[l]}},\overline{\bbq_\ell})\ne 0$. By the definition of intersection complex, $(X,w)$ must lies in $\overline{\co_{\mu[l]}}$.  On the other side, due to the $\uG$-equivariance, we have $\co_{(X,w)}\subset\text{Supp}(\calh^{d_\mu}_{(X,w)}\text{IC}
(\overline{\co_{\mu[l]}},\overline{\bbq_\ell}))$. Hence we have
$$\dim \text{Supp}(\calh^{d_\mu}_{(X,w)}\text{IC}
(\overline{\co_{\mu[l]}},\overline{\bbq_\ell}))\geq \dim\co_{(X,w)}.$$
By the axioms of perverse sheaves (see \cite{BBD}, or \cite[\S12.16]{jan2}), $d_\mu\leq 0$. Combining two sides, we finally have $\co_{\lambda[q]}=\co_{\mu[l]}$. Hence, as the first step we obtain
\begin{align}\label{eq: top ell coh}
H^{d}(\cb_{(X,w)},\overline{\bbq_\ell})\cong
V_{(\overline{\co_{\lambda[q]}},\overline{\bbq_\ell})}\otimes
    \overline{\bbq_\ell}.
    \end{align}
Next, recall that $\cb_{(X,w)}=\cb_X$, and all Springer fibers have affine pavings (see \cite{Spa2} or \cite[11.5]{jan2}). Hence $\dim H^{2i}(\cb_X,\overline{\bbq_\ell})$ is equal to the number of parts in the paving isomorphic to the affine space $\textbf{A}^i$ of dimension $i$, and $H^{2i+1}(\cb_X,\overline{\bbq_\ell})=0$ (see \cite[VI.4.20 and VI.11.1]{Mi} or \cite[\S12]{jan2}). However,  by Lemma \ref{thm: 4.3}(2) $\dim\cb_{(X,w)}={1\over2}(n^2-n-\dim\co_X)$. So $d= 2\dim\cb_{(X,w)}+q > 2\dim\cb_{(X,w)}$ if $q>0$. Hence $H^{d}(\cb_{(X,w)},\overline{\bbq_\ell})=0$ if $q>0$. Correspondingly,
 $V_{(\overline{\co_{\lambda[q]}},\overline{\bbq_\ell})}$ must be zero in \eqref{eq: top ell coh} as long as $q>0$. So the claim \eqref{eq: claim ic decomp} is proved.

Finally,  the $\uG$-equivariance of the decomposition \eqref{eq: ic decomp} has been assured in the arguments for \eqref{eq: Lu-Sp dec}. Summing up, we obtain the desired decomposition  \eqref{eq: ic decomp}.

(3) The last part is a direct consequence of the second one along with the classical springer theory (see \cite{Sho} or \cite{jan2}). As to the statement concerning perverse sheaves, one can be referred to \cite[1.3.6]{BBD} or \cite[\S12]{jan2}, \cite{Sho}.
\end{proof}

\begin{remark}
%(1) As to  $\text{IC}
 %   (\overline{\co_{\lambda[0]}},\overline{\bbq_\ell})[-
  %  (\dim\co_{\lambda[0]})]$, one can say more. \blue{(??????)}
(1) As to the multiplicity space $V_{(\overline{\co_{\lambda[0]}},
    \overline{\bbq_\ell})}$, its dimension is equal to the one of $V_{\overline{\co_\lambda}}$ in the classical Borho-MacPherson decomposition of the perverse sheaf on $\caln$ for $\gl_n$. This is because from the top etale cohomology computation \eqref{eq: top ell coh} and from the observation $\cb_{(X,w)}=\cb_X$, one has $V_{(\overline{\co_{\lambda[q]}},\overline{\bbq_\ell})}\otimes
    \overline{\bbq_\ell} \cong V_{(\overline{\co_{\lambda}},\overline{\bbq_\ell})}\otimes
    \overline{\bbq_\ell}$.

(2)  Thanks to  \cite[Theorem 1.1]{Ma} and its origination \cite[Theorem 4.7]{AH} and  \cite{AH2}, there is a general result on the existence of affine paving for all fibers of $\pie$. This implies that for any $(Y,u)\in \overline{\co_{\lambda[q]}}$, the fiber $\pie^{-1}(Y,u)$ has an affine paving.
%\end{remark}
%\subsubsection{Fibers of a general resolution $\pie$}

\end{remark}

\subsection{Enhanced version of De Concini-Procesi's theorem}  De Concini-Procesi described  the cohomology ring structure of Springer fibers   in \cite{dCP}. We present their result in the enhanced version.

Set $H^i(\cb_{(X,w)}):=H^i(\cb_{(X,w)},\bbz)$, the ordinary cohomology.  Let $H^\bullet(\cb_{(X,w)})$ denote the cohomology ring.
%\subsection{}

\begin{theorem}\label{prop: 4.6} Assume $(X,w)\in\ucaln$ is of type $\lambda[q]$. The following statements hold.
\begin{itemize}
\item[(1))] There is an isomorphism as algebras and $\frak{S}_n$-modules:
$$H^{\bullet}(\cb_{(X,w)})\cong H^{\bullet}(\cb_{X})\cong R_{\lambda}$$
where $R_\lambda$ is the coordinate ring of $\overline\co_{\lambda^{\textsf{t}}}\sqcap \hhh$,  the scheme-theoretic scheme intersection between $\overline{\co_{\lambda^{\textsf{t}}}}$ and the canonical torus $\hhh$ consisting of diagonal matrices.

\item[(2)] $\dim H^{\bullet}(\cb_{(X,w)})={n!\over a_1!\cdots a_t!}$ where $\lambda=(a_1\geq a_2\geq\cdots a_t> 0)$.
\end{itemize}
\end{theorem}

\begin{proof} By Lemma \ref{thm: 4.3}(1), $\cb_{(X,w)}=\cb_X$. The statements directly follow from  the classical results on Springer fibers (see \cite{dCP}, \cite[\S11.5]{jan2}, or \cite{Ta}).
\end{proof}

%\vskip5pt

%\subsection*{\textsc{Acknowledgement}} We thank some  anonymous referee for helpful comments.

\end{document}